\documentclass{siamart220329} 

\usepackage{amsmath}
\usepackage{amsfonts}
\usepackage{amssymb}
\usepackage{fullpage}
\usepackage{xcolor}
\usepackage{bm}
\usepackage{upgreek}
\usepackage{pgfplots}
\usepackage{subfig}
\usepackage{graphicx}
\usepackage{float}
\usepackage{relsize}

\usetikzlibrary{plotmarks}
\usetikzlibrary{spy}

\newcommand{\calV}{\mathcal{V}}
\newcommand{\calD}{\mathcal{D}}
\newcommand{\calS}{\mathcal{S}}

\newcommand\halfopen[2]{\ensuremath{[#1,#2)}}

\DeclareMathOperator*{\argmin}{arg\,min}

\newsiamremark{remark}{Remark} 
\crefname{Remark}{Remark}{Remark}

\def\ccm{Center for Computational Mathematics, Flatiron Institute, Simons Foundation,
  New York, New York 10010}

\def\nyu{Courant Institute of Mathematical Sciences,
  New York University, New York, New York 10012}
  
\title{A lightweight, geometrically flexible fast algorithm for the
evaluation of layer and volume potentials}
\author{Fredrik Fryklund%
  \thanks{\ccm\, \& \nyu\,
  ({\tt fredrik.fryklund@cims.nyu.edu}).}
    \and
Leslie Greengard%
  \thanks{\ccm\, \& \nyu\,
 ({\tt lgreengard@flatironinstitute.org}).}
     \and
Shidong Jiang%
  \thanks{\ccm\,
    ({\tt sjiang@flatironinstitute.org}).}
     \and
Samuel Potter%
  \thanks{\nyu\, \& Coreform\, 
  ({\tt sam@coreform.com}).}
    } 

\begin{document}
\maketitle
\begin{abstract}
Over the last two decades, several fast, robust, and high-order accurate methods
have been developed for solving the Poisson equation in complicated geometry
using potential theory. In this approach, rather than discretizing the 
partial differential equation itself, 
one first evaluates a volume integral to account for the source distribution within
the domain, followed by solving a boundary integral equation to impose
the specified boundary conditions. 
  
Here, we present a new fast algorithm which is easy to
implement and compatible with virtually any discretization technique, including
unstructured domain triangulations, such as those used in 
standard finite element or finite volume methods. 
Our approach combines earlier work on potential 
theory for the heat equation, asymptotic analysis, the nonuniform fast 
Fourier transform (NUFFT), and the dual-space multilevel kernel-splitting 
(DMK) framework.
It is {insensitive to flaws in the triangulation}, 
permitting not just nonconforming elements, but 
arbitrary aspect ratio triangles, 
gaps and various other degeneracies. On a single CPU core, the 
scheme computes the solution at a rate comparable to that of the 
fast Fourier transform (FFT) in 
work per gridpoint.
\end{abstract}
\section{Introduction}
In this paper, we consider the solution of the Poisson equation
\begin{align}
 \label{eq:poissoneq}
-\Delta u(\bm{x}) &= f(\bm{x}) \quad \text{for }\bm{x} \in \Omega,
\end{align}
where $\Omega \subset \mathbb{R}^{2}$ is either an interior or exterior
domain,
subject to either Dirichlet or Neumann boundary conditions:
\begin{equation}
\label{dbc}
u(\bm{x}) = g(\bm{x}) \quad \text{for } \bm{x} \in \partial\Omega,
\end{equation}
or
\begin{equation}
\label{nbc}
\frac{\partial u}{\partial \nu}(\bm{x}) = 
g(\bm{x}) \quad \,\text{for } \bm{x} \in \partial\Omega,
\end{equation}
respectively.
Here, $u$ is an unknown function, 
$\frac{\partial u}{\partial \nu}$ denotes its outward normal derivative,
$f$ is a smooth source density in $\Omega$, and
$g$ is the specified boundary data.
We will focus on integral equation methods, 
and will seek a representation
of the solution as a combination of volume and layer potentials.
The volume potential $\mathcal{V}[f](\bm{x})$ is defined by 
\begin{equation}
\label{eq:volpotposgreen}
    \mathcal{V}[f](\bm{x}) \equiv \int_{\Omega}\! G(\bm{x}-\bm{x'})f(\bm{x'})\,\mathrm{d}\bm{x'}
\quad \text{for }\bm{x}\in\Omega,
\end{equation}
where 
\begin{equation}
\label{eq:GreenPoisson}
    G(\bm{x}) = 
    -\frac{1}{2\pi}\log\|\bm{x}\|,
\end{equation}
is the fundamental solution for the Poisson equation in two dimension,
with $\|\bm{x}\|$ denoting the Euclidean norm of $\bm{x}\in\mathbb{R}^{2}$. 
Single and double layer potentials are defined by
\begin{equation}
\label{eq:slpot}
    \mathcal{S}[\sigma](\bm{x}) \equiv \int_{\partial\Omega} \! G(\bm{x}-\bm{x'}) \sigma(\bm{x'})\,\mathrm{ds}_{\bm{x'}} 
\quad \text{for }\bm{x}\in\Omega,
\end{equation}
and
\begin{equation}
\label{eq:dlpot}
    \mathcal{D}[\mu](\bm{x}) \equiv \int_{\partial\Omega}\!
\frac{\partial G}{\partial {\nu}_{\bm{x'}}} (\bm{x}-\bm{x'})
\mu(\bm{x'})\,\mathrm{ds}_{\bm{x'}}
\quad \text{for }\bm{x}\in\Omega,
\end{equation}
respectively. 
Here, $\bm{\nu}(\bm{x'})=(\nu_{1}(\bm{x'}),
\nu_{2}(\bm{x'}))$ 
is the outward unit normal at a point $\bm{x'}\in\partial\Omega$, and
$\frac{\partial G}{\partial {\nu}_{\bm{x'}}}(\bm{x}-\bm{x'})$ 
denotes the normal derivative of the fundamental solution at 
$\bm{x'}$.
The functions $\sigma(\bm{x'})$ and $\mu(\bm{x'})$ are supported on the 
boundary $\partial \Omega$, and will be referred to as single and double 
layer densities. Note that 
$\mathcal{S}[\sigma](\bm{x})$ and $\mathcal{D}[\mu](\bm{x})$ define
harmonic functions in $\Omega$ while 
$\mathcal{V}[f](\bm{x})$ satisfies the Poisson equation.
Letting 
$\bm{x} = \bm{x'} \pm r \bm{\nu}(\bm{x'})$, 
it is well-known that
the double layer potential satisfies the jump relations
\begin{equation} 
\label{dlpjump}
\lim_{r \rightarrow 0^+}
\calD[\mu](\bm{x'} \pm r \bm{\nu}(\bm{x'}) ) 
 =
\calD[\mu](\bm{x'}) \pm \frac12 \mu(\bm{x'}).
\end{equation}

For the interior Dirichlet problem in a simply connected region,
standard potential theory \cite{guenther1988partial,kelloggpot,mikhlin,stakgold} suggests a representation for the solution $u$ of the form
\begin{equation}
    u(\bm{x}) = \mathcal{V}[f](\bm{x}) + \mathcal{D}[\mu](\bm{x})
\quad \text{for }\bm{x}\in\Omega,
\label{dir_rep}
\end{equation}
where $\mu$ is unknown and to be determined by the boundary condition.
For the Neumann problem, we will use Green's representation formula:
\begin{equation}
    u(\bm{x}) = \mathcal{V}[f](\bm{x}) + 
\mathcal{S} \left[\frac{\partial u}{\partial {\nu}} \right] (\bm{x})
- \mathcal{D}[u] (\bm{x})
\quad \text{for }\bm{x}\in\Omega.
\label{neu_rep}
\end{equation}
Taking the limits of the representations 
\eqref{dir_rep} and \eqref{neu_rep}
as $\bm{x}$ approaches a point on the boundary leads to boundary integral equations,
which will be described in \cref{s:bie}.

An important feature of the potential theory approach is that 
$\mathcal{V}[f](\bm{x})$ is
an integral transform over the volumetric data; that is, no 
unknowns are introduced in the domain interior. 
The difference $v(\bm{x})= u(\bm{x}) - \mathcal{V}[f](\bm{x})$ is
harmonic (satisfying the Laplace equation). Thus, one can think of 
the layer potential contributions
in \eqref{dir_rep} and \eqref{neu_rep} as representing a {\em harmonic 
correction} to $\mathcal{V}[f](\bm{x})$, needed to enforce \eqref{dbc} or
\eqref{nbc}.  Any ``particular solution" that satisfies the 
inhomogeneous equation \eqref{eq:poissoneq} can be used in place of 
$\mathcal{V}[f](\bm{x})$.

We do not seek to review the literature here, but 
briefly summarize the various integral equation methods that are currently
available to put our new solver in context.
First, if the right-hand side $f(\bm{x})$ is not too complicated, 
one can attempt to find a 
particular solution $F(\bm{x})$ such that 
\begin{equation}
 -\Delta F(\bm{x}) = f(\bm{x}). 
\label{antilapdef}
\end{equation}
The function $F(\bm{x})$ in \eqref{antilapdef} is called an {\em anti-Laplacian} of $f$ and one can represent $u(\bm{x})$ in the form
$u(\bm{x})= v(\bm{x}) + \mathcal{V}[f](\bm{x})$ where $v(\bm{x})$
and represented via layer potentials
(see \cref{s:bie}).
Using a global anti-Laplacian in this manner is sometimes called the
{\em dual reciprocity method} \cite{brebbia}.
Unfortunately, it is only practical when $f(\bm{x})$ has a 
simple analytical representation.
To address problems where the volumetric data $f(\bm{x})$ is represented 
on a more complex data structure, fast algorithms are required, as the naive
cost of computing a volume integral with $N$ discretization points at the
same $N$ locations is of the order $O(N^2)$.
For uniform grids in complicated geometry, careful attention to finite difference 
approximations at points near the domain boundary was shown to permit fast solvers based on the fast Fourier transform (FFT) to compute volume potentials with the order of accuracy
of the underlying finite difference approximation 
\cite{Mayo,MGM}. These are related to immersed boundary 
and immersed interface methods \cite{iim_leveque,peskin_2002}.
Subsequently, {\em adaptive} fast algorithms for the direct computation of 
volume integrals began to emerge in the 1990s (see, for example, 
\cite{brandt1990jcp,vfmm,greengard1996direct,langston2011,malhotra_biros_2015}). 
These methods generally assume the right-hand side is provided on an adaptive
quadtree or octtree data structure (for 2D and 3D problems, respectively) 
without a complicated boundary. 

\begin{remark}
We will refer to fast algorithms for computing volume potentials 
on adaptive quad or octtrees (but without a complicated boundary)
as {\em box codes}. 
\end{remark}

When the source distribution in an irregular region is nonuniform,
at least three distinct approaches are under active investigation. These
include: (1) function extension methods that enable 
the straightforward application of box codes 
\cite{askham2017,bruno,fryklund2018partition,fryklund2020integral,epstein2023accurate,fg2023,STEIN2016252,STEIN2017155}, (2)
composite or overlapping grid methods which 
use a box code for leaf nodes away from the boundary, but a different data 
structure near the boundary
\cite{shravan_hai_volint,functionintension},
and (3) domain triangulation methods, which modify the fast multipole method
(FMM) but do not make use of box codes at all.
The latter approach has the advantage of making the solver compatible with the
discretizations used in finite element and finite volume methods, but 
has received less attention from the integral equation 
community. An early exception is \cite{russostrain} which coupled the FMM
with adaptive quadrature to compute $\mathcal{V}[f](\bm{x})$.
Recently, Shen and Serkh \cite{serkhvolint} 
developed a high-order accurate FMM that computes an 
anti-Laplacian on each triangle and enforces global continuity by means 
of layer potentials on triangle boundaries. This extends the box code
approach of \cite{greengard1996direct} to general geometries with good 
performance.
See also the recent interpolatory method of \cite{anderson2}.

Since there are now many linear
scaling schemes for computing volume integrals, the reader might well ask 
why we seek to develop a new approach. The answer is that
there is still a need for a fast solver that is
{\em easy to implement}, {\em robust},
{\em highly efficient}, and {\em compatible with virtually any discretization}.
In this paper, we attempt to
meet these criteria, achieving speeds near that of the FFT in work per gridpoint.
Our method, illustrated using unstructured triangulations, 
does not make use of FMM acceleration.
Instead, we blend the nonuniform FFT (NUFFT) with ``kernel splitting" and 
asymptotics, extending the 
dual-space multilevel kernel-splitting 
(DMK) framework
of \cite{dmk} to complicated geometries.  Moreover,  
we show that both volume and layer potentials can be evaluated 
accurately with the new scheme in a manner that is
insensitive to flaws in the triangulation, 
permitting nonconforming and arbitrary aspect ratio elements, 
gaps and various other degeneracies 
(see \cref{fig:exrobustmesh}).

An essential ingredient in our method is the use of asymptotic
expansions, described in \cref{s:local}. 
For layer potentials, this extends the
ideas presented 
in \cite{gs:fhp,li2009sisc,strainfpt,wang2019acom}.
Similar asymptotic analysis 
has also been carried out by Beale {\em et al.}
(see, for example, 
\cite{bealetlupova2024,bealetlupova2024b,beale2016}), although used
in a somewhat different manner. 
As far as we are aware, the use of asymptotic expansions for volume
integrals, as described below, has not been considered before.
(See also \cref{rmk:asymp}.)
\section{Mathematical preliminaries} \label{sec-preliminaries}
Our fast algorithm for the evaluation of volume and layer potentials is
naturally motivated by first
observing that the volume integral $\calV[f](\bm{x})$ in
\cref{eq:volpotposgreen} can be computed
as the steady state limit of the heat equation \cite{friedman1964,pogorzelski}
\begin{equation}
 u_t(\bm{x}) -  \Delta u(\bm{x}) = f(\bm{x}). 
\label{eq:heat}
\end{equation}
We may view this limit as the solution at time $t=0$ for the equation
\eqref{eq:heat} with zero initial data at time $t= -\infty$.
Using the fundamental solution for the heat equation 
in two dimensions
\[ K(\bm{x},t) = \frac{e^{-\|\bm{x}\|^2/4t}}{4 \pi t}, \]
we write this alternative representation as a {\em volume heat potential}
(where large time corresponds to a large time in the past):
\begin{equation}
\calV[f](\bm{x}) = 
\int_0^{\infty}\! \int_\Omega\!  \frac{e^{-\|\bm{x-x'}\|^2/4t}}{4 \pi t} \, 
f(\bm{x'}) \,\mathrm{d}\bm{x'} \, \mathrm{d}t.
\label{heatequiv}
\end{equation}
This is the function we will compute but divided into three contributions.  We let 
\begin{equation}
\mathcal{V}[f](\bm{x}) =  \mathcal{V}_L[f](\bm{x})  + 
\mathcal{V}_{NH}[f](\bm{x})  + \mathcal{V}_{FH}[f](\bm{x}) \, ,
\label{heatdecomp}
\end{equation}
where 
\begin{equation}
  \begin{aligned}
    \mathcal{V}_L[f](\bm{x}) &=  
                               \int_0^{\delta_1}\! \int_\Omega\!
                               \frac{e^{-\|\bm{x-x'}\|^2/4t}}{4 \pi t} \, f(\bm{x'}) \,\mathrm{d}\bm{x'} \, \mathrm{d}t
                               , 
                               \quad \mathcal{V}_{NH}[f](\bm{x}) \! \! \! \! &=  
                                                                               \int^{\delta_2}_{\delta_1} \! \int_\Omega \!
                                                                               \frac{e^{-\|\bm{x-x'}\|^2/4t}}{4 \pi t} \, f(\bm{x'}) \,\mathrm{d}\bm{x'} \, \mathrm{d}t, \\
    \mathcal{V}_{FH}[f](\bm{x}) &=  
                                  \int^{\infty}_{\delta_2} \! \int_\Omega \!
                                  \frac{e^{-\|\bm{x-x'}\|^2/4t}}{4 \pi t} \, f(\bm{x'}) \,\mathrm{d}\bm{x'} \, \mathrm{d}t. \\
    \label{heatdecompdefs}
  \end{aligned}
\end{equation}

$\calV_L$ accounts for contributions that are {\em local} in time,
$\calV_{NH}$ accounts for contributions from the {\em near history}, and 
$\calV_{FH}$ accounts for contributions from the {\em far history}.
Such a decomposition was proposed in 
\cite{gs:fhp} for the rapid evaluation of layer heat potentials, 
and used in the manner suggested above in \cite{strainfpt} for harmonic 
potentials of the form
\eqref{eq:slpot} and \eqref{eq:dlpot}. 
In those papers, 
the near and far history were treated together as the ``history" part:
\begin{equation}
\label{historysub}
\mathcal{V}_H[f](\bm{x}) = 
\mathcal{V}_{NH}[f](\bm{x})  + \mathcal{V}_{FH}[f](\bm{x}) \, .
\end{equation}

\begin{remark}
The DMK framework of
\cite{dmk} was
introduced as an alternative to box codes (among other things). It is a
hierarchical, adaptive, Fourier-based method that, in the present context,
can be viewed as using $O(\log N)$ levels of subdivision with respect to the 
time variable together with an adaptive quadtree data structure.
Our work here combines kernel splitting 
with asymptotics to evaluate the convolution of the harmonic Green's 
function with nonsmooth distributions
(i.e., layer and volume potentials in complicated geometry).
For the sake of simplicity, we restrict ourselves to quasi-uniform 
discretizations and a two-level implementation, accelerated by the NUFFT
instead of the full DMK machinery. We will return to this topic in
\cref{s:discussion}.
\end{remark}

It should be noted that
the decomposition of a volume heat potential in two pieces,
namely $\mathcal{V}[f](\bm{x}) = \mathcal{V}_L[f](\bm{x}) 
+ \mathcal{V}_H[f](\bm{x})$ is essentially the basis for Ewald summation
\cite{ewald},
although originally developed for discrete sources in three dimensions
with periodic boundary conditions, rather than continuous distributions
in free space.  In Ewald's treatment, modified
for the two dimensional setting, 
the basic idea is that of  {\em kernel splitting}:
expressing the Green's function as
\begin{equation}
\label{eq:splitting}
    G(\bm{x}) = -\frac{1}{2\pi}\log\|\bm{x}\| = G_H[\delta](\bm{x}) 
+ G_L[\delta](\bm{x}) \, , 
\end{equation}
with 
\begin{equation}
\label{eq:splitting2}
G_H[\delta](\bm{x}) = \left[ -\frac{1}{2\pi}\log\|\bm{x}\| 
    -\frac{1}{4\pi}E_1\! \left( \frac{\|\bm{x}\|^2}{4\delta} \right) 
\right] \quad , \quad 
G_L[\delta](\bm{x}) =  \frac{1}{4\pi}
E_1\!\left( \frac{\|\bm{x}\|^2}{4\delta} \right) \, ,
\end{equation}
where 
\[ E_1(x) = \int_1^{\infty} \frac{e^{-tx}}{t} \, \mathrm{d}t \]
is the exponential integral function \cite{dlmf}.
The connection between the decomposition in terms of heat potentials
and kernel splitting is established by
the identity 
\begin{equation}
 \int_0^{\delta} \!
\frac{e^{-\|\bm{x}\|^2/4t}}{4 \pi t} \, \mathrm{d}t =
\frac{1}{4\pi}E_1 \! \left( \frac{\|\bm{x}\|^2}{4\delta} \right)
\label{eiformula}
\end{equation}
for $\bm{x} \neq 0$.
We refer the reader to 
\cite{klinteberg2017,gs:fhp,dmk,li2009sisc,palssontornberg,shamshirgar2021jcp,strainfpt} 
for further discussion.

It remains only to understand how each of the terms in 
\eqref{heatdecomp} is to be computed. Simply stated,
$\mathcal{V}_L[f](\bm{x})$ is local in both space and time
and will be treated using asymptotics, while 
$\mathcal{V}_{NH}[f](\bm{x})$ and $\mathcal{V}_{FH}[f](\bm{x})$
will be computed using Fourier methods. 

\begin{definition}
For a function $f(\bm{x})$ in two dimensions, we define its Fourier transform
by 
\[ \hat{f}(\bm{k}) = \int_{\bm{x} \in \mathbb{R}^2} e^{-i\bm{k} \cdot \bm{x}} f(\bm{x}) \, \mathrm{d}\bm{x},
\]
where $\bm{k} \in \mathbb{R}^2$. The function $f(\bm{x})$ is recovered from 
$\hat{f}(\bm{k})$ by the inverse transform
\[ f(\bm{x}) = \frac{1}{ (2 \pi)^2} 
\int_{\bm{k} \in \mathbb{R}^2} e^{i\bm{k} \cdot \bm{x}} \hat{f}(\bm{k}) \, \mathrm{d}\bm{k}.
\]
\end{definition}

\begin{lemma} \label{nh_lemma}
For a function $f(\bm{x})$ compactly supported in $\Omega$,
the near history component has the Fourier representation
\begin{equation}
\mathcal{V}_{NH}[f](\bm{x}) =  
\int_{\bm{k} \in \mathbb{R}^2}
\frac{ e^{- \|\bm{k}\|^2 \delta_1}- e^{- \|\bm{k}\|^2 \delta_2}}{\|\bm{k}\|^2}
e^{i \bm{k \cdot x}} \hat{f}(\bm{k}) \, \mathrm{d}\bm{k}.
\label{vnearkspace}
\end{equation}
\end{lemma}
\begin{proof}
This result follows from the fact that the Fourier transform of
${e^{-\|\bm{x}\|^2/4t}}/(4 \pi t)$ is $e^{- \|\bm{k}\|^2 t}$,
the convolution theorem, and
integration in time.
\end{proof}

\vspace{.2in}

Note that the integrand in \eqref{vnearkspace} is rapidly decaying in
Fourier space and that it is infinitely differentiable, since 
the expression
\[ 
\frac{ e^{- \|\bm{k}\|^2 \delta_1}- e^{- \|\bm{k}\|^2 \delta_2}}{\|\bm{k}\|^2}
\]
has a convergent power series and $f(\bm{x})$ is assumed to have compact support.

The far history is a little more delicate, as 
its Fourier transform 
has a singularity of the order $1/\|\bm{k}\|^2$ at the origin
and its rigorous analysis would involve consideration of tempered distributions. 
Using the method of 
\cite{vico2016jcp}, however, we have the following lemma.

\begin{lemma}
Suppose that $f(\bm{x})$ is compactly supported in $\Omega \subset [-1,1]^2$.
Then, for $\bm{x}$ with $\| \bm{x}\| \leq 1 $,
$\mathcal{V}_{FH}[f](\bm{x})$ has the Fourier
representation
\begin{equation}
\mathcal{V}_{FH}[f](\bm{x}) \approx  
\int_{\bm k \in \mathbb{R}^2}
e^{- \|\bm{k}\|^2 \delta_2}
W(\bm{k})
e^{i \bm{k \cdot x}} \hat{f}(\bm{k}) \, \mathrm{d}\bm{k},
\label{vfarkspace}
\end{equation}
where
\begin{equation}
\label{Wdef}
W(\bm{k}) = \left[ \frac{1 - J_0(2 \sqrt{2}\, \|\bm{k}\|)}{\|\bm{k}\|^2} - 
\frac{2 \sqrt{2} \log(2 \sqrt{2}) \, J_1( 2\sqrt{2} \, \|\bm{k}\|)}{ \| \bm{k}\|} 
\right],
\end{equation}
with $J_0$ and $J_1$ denoting Bessel functions of the first kind. The error 
in \eqref{vfarkspace} is of the 
order $O(e^{-d^2/\delta_2})$, where $d$ is the distance of the support of $f(\bm{x})$ from the enclosing
box $[-1,1]^2$. 
\end{lemma}
\begin{proof}
The result follows from the convolution theorem, as in \cref{nh_lemma},
and integration in time, but replacing the heat kernel with
\[  
\frac{e^{-\|\bm{x}\|^2/4t}}{4 \pi t} \, {\rm rect}(\|\bm{x}\|/4),
\]
where ${\rm rect}(x)$ is the characteristic function for the unit interval,
\[
{\rm rect}(x) =  \left\{
\begin{array}{cc}
 1 & {\rm for}\ |x|<1/2, \\
 0 & {\rm for}\ |x|>1/2.
\end{array}
\right.
\]
This modification gives rise to the term $W(\bm{k})$ in \eqref{Wdef}, 
in place of $1/\|\bm{k}\|^2$, in the integrand.
We omit the detailed calculation and refer the interested reader to 
\cite{dmk,vico2016jcp}. 
The error is due to the fact that the mollification of $f(\bm{x})$ involves exponentially decaying
leakage outside the unit box. This can be made arbitrarily small by slightly increasing the 
box size and rescaling.
\end{proof}

\begin{lemma}
    The function $W(\bm{k})$ in \cref{Wdef} is smooth with limiting value
    \begin{equation}
        \lim\limits_{\|\bm{k}\| \rightarrow 0} W(\bm{k}) = \frac{(1-2\log(2\sqrt{2}))(2\sqrt{2})^{2}}{4}.
    \end{equation}
\end{lemma}

As for \eqref{vnearkspace},
the integrand in \eqref{vfarkspace} is clearly exponentially decaying.
Taylor expansion of the Bessel functions at the origin shows that
the integrand is also infinitely differentiable. 
As a result, both $\mathcal{V}_{NH}[f](\bm{x})$ and
$\mathcal{V}_{FH}[f](\bm{x})$ can be computed using the 
trapezoidal rule for quadrature with 
spectral accuracy \cite{trefethentrap}.

There are several novel features of the method described in this paper.
First, we have derived new asymptotic expansions for 
$\calV_L[f]$ when $f$ is smooth in $\Omega$ but discontinuous
as a function in $\mathbb{R}^2$ as well as high order expansions for
$\calS_L[\sigma]$ and $\calD_L[\sigma]$ for points off surface.
Second, we have developed a new telescoping series
for the local part of layer (and volume) potentials that can be combined 
with our asymptotic techniques to yield arbitrary order accuracy.
Third, our division of the history part into the near and far components
reduces the cost of the Fourier transforms 
by a significant (albeit constant) factor
(see \cref{r:twolevel} below).
\subsection{Layer potentials}
The equivalence with steady state limits of the heat equation
extends to single and double layer potentials as well.
That is, we may write

\begin{equation}
\mathcal{S}[\sigma](\bm{x}) \ = \ 
\int_0^{\infty} \!\int_{\partial \Omega} \!
\frac{e^{-\|\bm{x-x'}\|^2/4t}}{4 \pi t} \, 
\sigma(\bm{x'}) \,\mathrm{ds}_{\bm{x'}} \, \mathrm{d}t \ \ = \ \ 
 \mathcal{S}_L[\sigma](\bm{x})  + 
\mathcal{S}_{NH}[\sigma](\bm{x})  + \mathcal{S}_{FH}[\sigma](\bm{x}), 
\label{sheatdecomp}
\end{equation}
and
\begin{equation}
\mathcal{D}[\mu](\bm{x}) =
\int_0^{\infty} \!\int_{\partial \Omega}
\frac{\partial}{\partial {\nu}_{\bm{x'}}}
\frac{e^{-\|\bm{x-x'}\|^2/4t}}{4 \pi t} \, \mu(\bm{x}') \,\mathrm{ds}_{\bm{x}'} \, \mathrm{d}t 
\ = \ 
  \mathcal{D}_L[\mu](\bm{x})  + 
\mathcal{D}_{NH}[\mu](\bm{x})  + \mathcal{D}_{FH}[\mu](\bm{x}) \, .
\label{dheatdecomp}
\end{equation}
Here,
\begin{align}
\mathcal{S}_L[\sigma](\bm{x}) &=  
\int_0^{\delta_1}\! \int_{\partial \Omega}\!
\frac{e^{-\|\bm{x-x'}\|^2/4t}}{4 \pi t} \, \sigma(\bm{x'}) \,\mathrm{ds}_{\bm{x'}} \, \mathrm{d}t, 
\label{sheatloc} \\
\mathcal{D}_L[\mu](\bm{x}) &=  
\int_0^{\delta_1} \! \int_{\partial \Omega}
\frac{\partial}{\partial {\nu}_{\bm{x'}}}
\frac{e^{-\|\bm{x-x'}\|^2/4t}}{4 \pi t} \, \mu(\bm{x}') \,\mathrm{ds}_{\bm{x}'} \, \mathrm{d}t, 
\label{dheatloc}
\end{align}
and so on.  From the preceding discussion, the proof of the following two
lemmas is straightforward.
\begin{lemma}
The near history of the single layer potential has the Fourier representation
\begin{equation}
\mathcal{S}_{NH}[\sigma](\bm{x}) =  
\int_{\bm{k} \in \mathbb{R}^2}\!
\frac{ e^{- \|\bm{k}\|^2 \delta_1}- e^{- \|\bm{k}\|^2 \delta_2}}{\|\bm{k}\|^2}
e^{i \bm{k \cdot x}} \hat{\sigma}(\bm{k}) \, \mathrm{d}\bm{k},
\label{snearkspace}
\end{equation}
where 
\[ \hat{\sigma}(\bm{k}) = \int_{\partial\Omega} 
\sigma(\bm{x}) e^{-i \bm{k \cdot x}} \, \mathrm{ds}_{\bm{x}}.
\]
The near history of the double layer potential has the Fourier representation
\begin{equation}
\mathcal{D}_{NH}[\mu](\bm{x}) =  
\int_{\bm{k} \in \mathbb{R}^2}\!
\frac{ e^{- \|\bm{k}\|^2 \delta_1}- e^{- \|\bm{k}\|^2 \delta_2}}{\|\bm{k}\|^2}
e^{i \bm{k \cdot x}} \,
[ ik_1 \widehat{\mu \, \nu_1}(\bm{k}) +
 ik_2 \widehat{\mu \, \nu_2}(\bm{k}) ]
 \, \mathrm{d}\bm{k},
\label{dnearkspace}
\end{equation}
where 
\[ \widehat{\mu \, \nu_j}(\bm{k}) = \int_{\partial \Omega} 
\mu(\bm{x}) \nu_j(\bm{x}) e^{-i \bm{k \cdot x}} \, \mathrm{ds}_{\bm{x}},
\]
$\bm{k} = (k_1,k_2)$,
and $(\nu_1(\bm{x}),\nu_2(\bm{x}))$ are the components of the normal at 
$\bm{x} \in \partial\Omega$.
\end{lemma}

\begin{lemma}
The far history of the single layer potential has the Fourier representation
\begin{equation}
\mathcal{S}_{FH}[\sigma](\bm{x}) =  
\int_{\bm{k} \in \mathbb{R}^2}
e^{- \|\bm{k}\|^2 \delta_2}\,
W(\bm{k})
e^{i \bm{k \cdot x}} \hat{\sigma}(\bm{k}) \, \mathrm{d}\bm{k},
\label{sfarkspace}
\end{equation}
where $W(\bm{k})$ is given by \eqref{Wdef}.
The far history of the double layer potential has the Fourier representation
\begin{equation}
\mathcal{D}_{FH}[\mu](\bm{x}) =  
\int_{\bm{k} \in \mathbb{R}^2}
e^{- \|\bm{k}\|^2 \delta_2}\,
W(\bm{k})
e^{i \bm{k \cdot x}} \,
[ ik_1 \widehat{\mu \, \nu_1}(\bm{k}) +
 ik_2 \widehat{\mu \, \nu_2}(\bm{k}) ]
 \, \mathrm{d}\bm{k}.
\label{dfarkspace}
\end{equation}
\end{lemma}
\subsection{Discretization} \label{s:discrete}
In order to describe our algorithm in detail, we need to fix some 
conventions about discretization and the user interface.

\begin{figure} 
\centering
\includegraphics[width=5in]{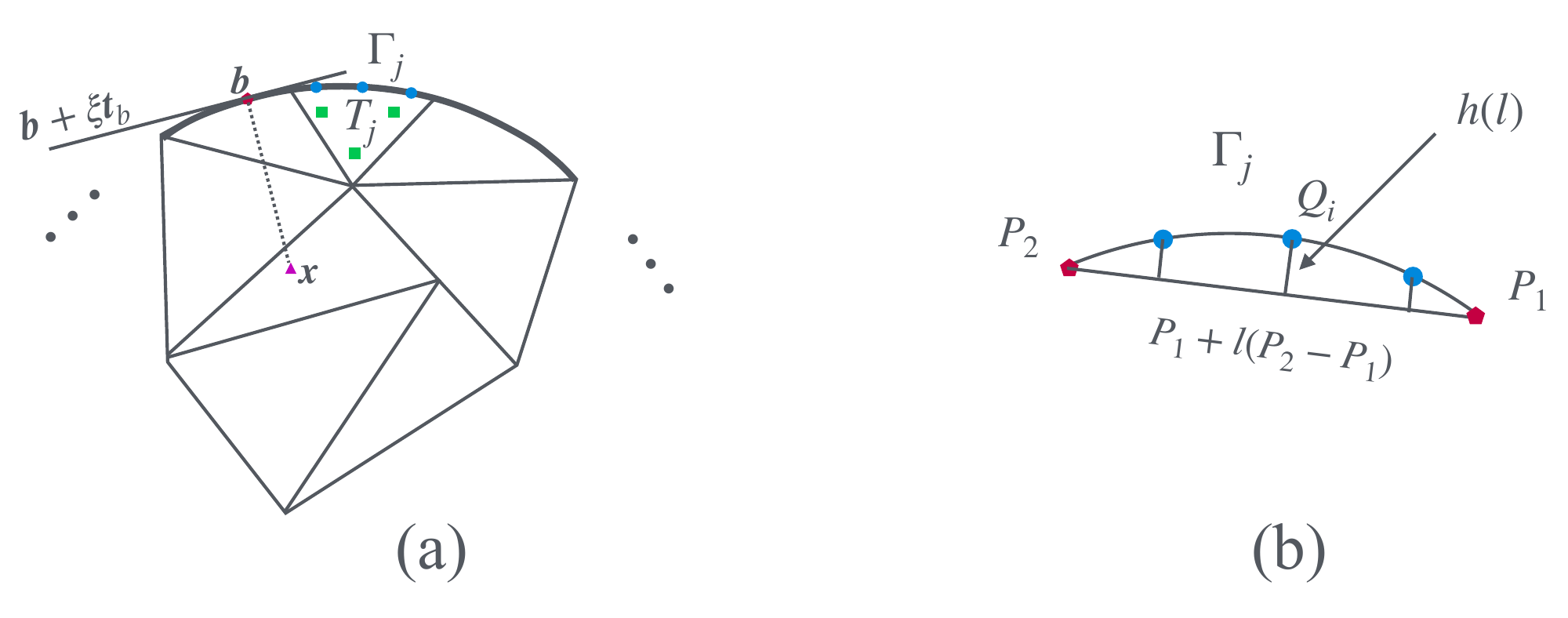}
\caption{(a) A portion of a domain composed of triangular elements.
The triangle  $T_j$ is a {\em boundary triangle} with curved boundary
segment $\Gamma_j$ described by at least three points (indicated by blue
circles). The green squares in $T_j$ 
indicate the locations of quadrature nodes for an integration rule over the 
triangle.  For an interior point $\bm{x}$, we
denote by $\bm{b}$ the closest point on the boundary.
In our asymptotic formulas, we re-express the curve and layer potential 
densities in terms of the parameter $\xi$, the 
excursion in the tangent direction $\bm{t}_b$. 
(b) Given the points defining a curved boundary segment (or {\em chunk}) 
$\Gamma_j$,
we reparametrize $\Gamma_j$ in terms of chordal arclength $l$ 
$-$ that is,
the distance along the chord connecting the first endpoint $P_1$ with the
second endpoint $P_2$. 
$h(l)$ is the excursion of a boundary node $Q_i$ in the direction of the 
outward normal to the chord.
When we wish to refine/oversample the user-provided 
chunk, we do this in the chordal parametrization by interpolating
$h(l)$.
The additional work is negligible and the user interface is unaffected.
}
\label{fig-discrete}
\end{figure}

$\bullet$\ 
We assume that the domain $\Omega$ has been approximated as a union
of triangles: $\Omega^{\ast} = \cup_{j=1}^M T_j$, and that each triangle $T_j$
is equipped with a $p$-point quadrature rule of order $q_v$, so that
\[ \int_{T_j}  g(\bm{x}) \, \mathrm{d}\bm{x} \approx \sum_{l=1}^p g(\bm{x}_l) \, w_l 
\]
for smooth functions $g(\bm{x})$. The total number of interior
degrees of freedom is denoted by $N = Mp$.

\vspace{.1in}

$\bullet$\ 
We will refer to the triangles whose sides form the boundary of $\Omega^{\ast}$ as
{\em boundary triangles}. For each triangle segment lying on the
boundary, we assume that it is described by a polynomial of 
degree $q_B$ with respect to some parametrization (see \cref{fig-discrete}).
For simplicity of presentation, we assume that $q_B \geq 2$, so that
we have at least a first order accurate estimate of curvature.
This requires at least $q_B+1$ points on each segment. 

\vspace{.1in}

$\bullet$\ 
We will refer to the boundary segments as {\em chunks},
and denote by $M_B$ the number of such chunks, so that
$\Gamma^{\ast} =  \partial \Omega^{\ast} = \cup_{j = 1}^{M_B} \Gamma_j$.
The total number of boundary degrees of freedom is
$N_B = M_B(q_B+1)$. Note that for a reasonably uniform discretization,
$N_B \approx \sqrt{N}$.
As we shall see in \cref{s:dyadicref}, we will have the need to oversample
each chunk on the fly. For this, we reparametrize the chunk in terms
of chordal arc length as illustrated in \cref{fig-discrete}. 

\vspace{.1in}

$\bullet$\ 
Let $\delta_1$ denote the local cutoff in defining 
$\calV_L[f]$, $\calS_L[\sigma]$, $\calD_L[\mu]$.
We assume that $\delta_1$ is sufficiently small that, for any 
interior point $\bm{x} \in \Omega$ whose distance from the boundary is
less than $12\sqrt{\delta_1}$, there is a unique point on $\Gamma^{\ast}$
closest to $\bm{x}$. That is, 
$\bm{b} =\argmin_{\bm{y}\in\Gamma^{\ast}}\|\bm{y}-\bm{x}\|_{2}$ is well-defined
and unique (see \cref{fig-discrete}).
(The value $12\sqrt{\delta}$ is chosen because the local part of the heat
kernel has decayed to less than $10^{-14}$ at that distance.)
We assume the availability of a procedure to determine 
$\bm{b}$. In our examples, we use the closest discretization node
to $\bm{x}$ as an initial guess, followed by two steps of Newton's
method to find $\bm{b}$ with high precision. Since this is a standard 
issue in computational geometry, independent of the solver,
we omit further details (except to note that we can make use of the 
chordal parametrization for this).

\vspace{.1in}

$\bullet$\ 
The parameter
$\delta_1$ also controls the rate of decay of the Fourier transform for
$\calV_{NH}[f]$, $\calS_{NH}[\sigma]$, $\calD_{NH}[\mu]$. From the
exponential decay of the integrand, 
once $k^{NH}_{max}$ satisfies $e^{-\left(k^{NH}_{max}\right)^2 \delta_1} < \epsilon$ we have that
\begin{equation}
\mathcal{V}_{NH}[f](\bm{x}) =
\int_{-k^{NH}_{max}}^{k^{NH}_{max}}
\frac{ e^{- \|\bm{k}\|^2 \delta_1}- e^{- \|\bm{k}\|^2 \delta_2}}{\|\bm{k}\|^2}
e^{i \bm{k \cdot x}} \hat{f}(\bm{k}) \, \mathrm{d}\bm{k}  + O(\epsilon) \, ,
\label{fouriernhtrunc}
\end{equation}
and the same holds for $\calS_{NH}[\sigma]$ and $\calD_{NH}[\mu]$.
(For the double layer potential, it is clear from 
\eqref{dnearkspace} that a slightly greater cutoff is needed, namely 
$k^{NH}_{max}$ must satisfy $k^{NH}_{max} e^{-\left(k^{NH}_{max}\right)^2 \delta_1} < \epsilon$.)

\vspace{.1in}

$\bullet$\ 
The parameter $\delta_2$ controls the rate of decay of 
the Fourier transform for
$\calV_{FH}[f]$, $\calS_{FH}[\sigma]$, $\calD_{FH}[\mu]$. 
Once $k^{FH}_{max}$ satisfies $e^{-\left(k^{FH}_{max}\right)^2 \delta_2} < \epsilon$, we have that
\begin{equation}
\mathcal{V}_{FH}[f](\bm{x}) =
\int_{-k^{FH}_{max}}^{k^{FH}_{max}}
e^{- \|\bm{k}\|^2 \delta_2}\,
W(\bm{k})
e^{i \bm{k \cdot x}} \hat{f}(\bm{k}) \, \mathrm{d}\bm{k} +
O(\epsilon) \, ,
\label{vfarkspacetrunc}
\end{equation}
with $W(\bm{k})$ given by \eqref{Wdef}.
The same holds for $\calS_{FH}[\sigma]$ and $\calD_{FH}[\mu]$.
(For the double layer potential, as above, a
slightly greater cutoff is needed, namely 
$k^{FH}_{max}$ must satisfy $k^{FH}_{max} e^{-\left(k^{FH}_{max}\right)^2 \delta_2} < \epsilon$.)

\vspace{.1in}

$\bullet$\ 
It remains to determine the number of discretization points needed for
the integrals in \eqref{fouriernhtrunc} and \eqref{vfarkspacetrunc}.
Let us assume, without loss of generality, that the 
domain of interest lies in $B = [-1,1]^2$.
For interior problems, this requires that $\Omega \subset B$.
For exterior problems, this requires that
$\Omega$, the support of $f(\bm{x})$, and any additional 
target points of interest all lie in $B$.
Assuming our domain is triangulated quasi-uniformly with $N$ interior points,
the average spacing between points, denoted by $\Delta x$, 
in either coordinate direction is of the order $O(1/\sqrt{N})$.
Letting $\delta_1 = O(1/N)$ and
letting  $n^{NH}_f$ denote the number of equispaced quadrature nodes used on
$[-k^{NH}_{max},k^{NH}_{max}]$ in \eqref{fouriernhtrunc}, the error
decays superalgebraically with $n^{NH}_f$, with rapid convergence beginning when 
$n_f = O(\sqrt{N})$. Thus, the total number of discretized Fourier modes, 
$N_F=(n^{NH}_f)^2$ is of the order $O(N)$.
To see this, observe that the maximum excursion between a source
and a target is $2$ in each dimension, so that there are at most 
$2 k^{NH}_{max}/(2\pi)$ oscillations in the integrand and Nyquist sampling
corresponds to $n^{NH}_f = 2k^{NH}_{max}/\pi$.
Since  $\delta_1 = O(1/N) = O(\Delta x^2)$, in order to achieve a precision
$\epsilon$, we must have 
$e^{-\left(k^{NH}_{max}\right)^2 \delta_1} = \epsilon$. Combining these facts,
\begin{equation}
n^{NH}_f \approx  \frac{\sqrt{\log(1/\epsilon)}}{\pi} \cdot O(1/\sqrt{\delta_1}) 
\approx
\frac{\sqrt{\log(1/\epsilon)}}{\pi} \cdot O(\sqrt{N}). 
\label{nfNHest}
\end{equation}
Similar estimates for $n^{NH}_f$ hold for $\calS_{NH}[\sigma]$ or $\calD_{NH}[\mu]$.
For the far history contribution \eqref{vfarkspacetrunc}, 
the kernel $W(\bm{k})$ introduces additional oscillations.
It is straightforward to show that, for Nyquist sampling,
\begin{equation}
n^{FH}_f \approx 
    2\sqrt{2}\, \frac{\sqrt{\log(1/\epsilon)}}{\pi} \cdot O(1/\sqrt{\delta_2}).
\label{nfFHest}
\end{equation}

We omit a detailed proof and refer the interested reader to 
\cite{dmk,vico2016jcp}.
Note that, once $\delta_2> 8\delta_1$, the total number of Fourier modes for the 
far history is smaller than for the near history, despite the fact that
the integrand is more oscillatory. In practice, 
we choose $\delta_2 = 32\delta_1$ for optimal performance.
\section{Rapid evaluation of the near and far history} \label{s:history}
The algorithm below accounts for the near and far history from any combination
of volume, single and double layer potentials using the 
Fourier representations from the preceding section. 
All computations are easily carried out using the NUFFT
\cite{finufft,finufftlib,nufft2,nufft3,nufft4,nufft6,potts2004}.

\vspace{.2in}

\noindent
\begin{center}
{\bf NUFFT computation of history part}
\end{center}
\vspace{.2in}

{{\bf Input}:} Suppose we are given a set of points $\bm{x}[j]$ and smooth 
quadrature weights $w[j]$ for $j=1,\ldots,N$ that discretize 
the support of $f(\bm{x})$ 
in the domain $\Omega$ (whether it is an interior or exterior problem). 
Suppose we are also given a set of points $\bm{x}_{\Gamma}[j']$ and weights
$w_\Gamma[j']$ for $j'=1,\ldots,N_B$ that discretize smooth functions
restricted to the boundary $\Gamma$.
Let $\delta_1 = C \Delta x^2$ where
$\Delta x$ is the average grid spacing in the discretization. (We
typically choose $C = 3$.) Let $\delta_2 = 32 \, \delta_1$. For a tolerance
$\epsilon$,
compute the smallest $k^{NH}_{max}$, $k^{FH}_{max}$ such that 
\[ k^{NH}_{max} e^{-\left(k^{NH}_{max}\right)^2 \delta_1} < \epsilon \quad , \quad
k^{FH}_{max} e^{-\left(k^{FH}_{max}\right)^2 \delta_2} < \epsilon.
\]

\vspace{.1in}

\noindent
{{\bf Step 1}:} 
Compute $n^{NH}_f$ and $n^{FH}_f$ according to \eqref{nfNHest} and
\eqref{nfFHest}, to achieve about twice the Nyquist sampling frequency:
\begin{equation}
n^{NH}_f \equiv  \frac{2 \sqrt{\log(1/\epsilon)}}{\pi \, \sqrt{\delta_1}}  
\quad , \quad
n^{FH}_f \equiv  \frac{4 \sqrt{2} \sqrt{\log(1/\epsilon)}}{\pi \, \sqrt{\delta_2}}.
\label{historytol}
\end{equation}

Let $(k^{NH}_{l,1},k^{NH}_{l,2}) = 
\bm{k}^{NH}[l]$ denote the uniformly spaced grid points on the 
square $[-k^{NH}_{max},k^{NH}_{max}]^2$ with $n^{NH}_f$ points in each linear
dimension. Let $N^{NH}_f = (n^{NH}_f)^2$ denote the total number of discrete
Fourier modes for the near history.
Similarly, let 
$(k^{FH}_{l,1},k^{FH}_{l,2}) = \bm{k}^{FH}[l]$ 
denote the uniformly spaced grid points on the 
square $[-k^{FH}_{max},k^{FH}_{max}]^2$ with $n^{FH}_f$ points in each linear
dimension. Let $N^{FH}_f = (n^{FH}_f)^2$ denote the total number of discrete
Fourier modes for the far history.

\vspace{.1in}

\noindent
{{\bf Step 2}:} Compute the discrete Fourier transform
\[
\begin{aligned}
\hat{u}_{NH}[l] = \qquad \qquad
& \sum_{j=1}^N e^{i \bm{k}[l] \cdot \bm{x}[j]} \left(f(\bm{x}[j]) w[j]\right) + \\
\qquad
& \sum_{j'=1}^{N_B} 
e^{i \bm{k}[l] \cdot \bm{x}_\Gamma[j]} 
\left( \sigma(\bm{x}_\Gamma[j']) w_\Gamma[j'] \right) + \\
  i k_{l,1} 
& \sum_{j'=1}^{N_B} 
e^{i \bm{k}[l] \cdot \bm{x}_\Gamma[j']} 
\left( \nu_1(\bm{x}_\Gamma[j']) \mu(\bm{x}_\Gamma[j']) w_\Gamma[j'] \right) + \\
  i k_{l,2} 
& \sum_{j'=1}^{N_B} 
e^{i \bm{k}[l] \cdot \bm{x}_\Gamma[j']} 
\left( \nu_2(\bm{x}_\Gamma[j']) \mu(\bm{x}_\Gamma[j']) w_\Gamma[j'] \right)
\end{aligned}
\]
where $\bm{k}[l] = (k_{l,1},k_{l,2}) = \bm{k}^{NH}[l]$
for $l = 1,\ldots,N^{NH}_f$. 

The NUFFT computes precisely such sums, requiring
$O( N_f \log N_f + N + N_B)$ work for $N_f$ Fourier modes,
and there are a number of open source packages for this purpose.
We use the high performance FINUFFT package \cite{finufft,finufftlib}.

Note that the volumetric and single layer contributions 
(involving $f(\bm{x})$ and $\sigma(\bm{x}_\Gamma)$) require one NUFFT call.
For the double layer potential
(involving $\mu(\bm{x}_\Gamma)$), two calls are required
because the third sum involves post-multiplication with
$i k_{l,1}$, and the fourth sum involves post-multiplication with
$i k_{l,2}$.

\vspace{.1in}

\noindent
{{\bf Step 3}:} 
Repeat with
$\bm{k}[l'] = \bm{k}^{FH}[l']$
for $l' = 1,\ldots,N^{FH}_f$ to obtain
$\hat{u}_{FH}[l']$.

\vspace{.1in}

\noindent
{{\bf Step 4}:} 
For every point $\bm{x}$ of interest, 
compute the near and far history contributions: 
\[
u^{NH}(\bm{x}) = 
 \sum_{l = 1}^{N_f^{NH}} 
\hat{u}_{NH}[l]   
\left[
\frac{ e^{- \|\bm{k}[l]\|^2 \delta_1}- e^{- \|\bm{k}[l]\|^2 \delta_2}}{\|\bm{k}[l]\|^2}
\right]
e^{-i \bm{k}[l] \cdot \bm{x}}
\]
where $\bm{k}[l] = \bm{k}^{NH}[l]$, and
\[
{u}^{FH}(\bm{x}) = 
 \sum_{l' = 1}^{N_f^{FH}} 
\hat{u}_{FH}[l'] \,  
W(\bm{k}[l'])\,
e^{-i \bm{k}[l'] \cdot \bm{x}}
\]
where $\bm{k}[l'] = \bm{k}^{FH}[l']$.
The point $\bm{x}$ can lie anywhere in the domain $\Omega$ or on its boundary,
 but within the enclosing box $B$. If $N_{tot}$ denotes the total number of such
target points, the cost is of the order
$O\left(N_f^{NH} \log N_f^{NH} + N_{tot}\right)$ +
$O\left(N_f^{FH} \log N_f^{FH} + N_{tot}\right)$ using the NUFFT.

\begin{remark} \label{r:twolevel}
A simpler algorithm would be to merge the near and far history parts together,
by evaluating the expression for the far history part alone with 
$\delta_2 = \delta_1$. This, however, would require setting 
\[
n^{FH}_f \equiv  \frac{4 \sqrt{2} \sqrt{\log(1/\epsilon)}}{\pi \, \sqrt{\delta_1}},
\]
with a computational cost that is approximately $8 = (2\sqrt{2})^2$ times more
expensive than the near history component above. By setting 
$\delta_2 > 8 \delta_1$, the additional burden is at most a factor of two, 
resulting in a net speedup by a factor of 4 
(as well as reduced storage requirements).
\end{remark}
\section{Rapid evaluation of local heat potentials}\label{s:local}
In this section, we present explicit asymptotic expansions for 
$\mathcal{S}_{L}[\sigma](\bm{x})$, $\mathcal{D}_{L}[\mu](\bm{x})$, and 
$\mathcal{V}_{L}[f](\bm{x})$ for orders up to $\delta^{5/2}$. 
Since we have fixed $\delta_1$ to be of the order $O(\Delta x^2)$, these 
formulas yield fifth order accuracy in $\Delta x$ without additional effort.
The target point $\bm{x}\in\mathbb{R}^{2}$ 
may either be on the boundary, in the interior of $\Omega$ or in its 
exterior $\Omega^E$. For points off boundary, as discussed in
\cref{s:discrete}, we assume there is a
unique closest point $\bm{b} = (b_{1},b_{2})$ on $\partial\Omega$. 
Letting  
$\bm{\nu}_{b}=\bm{\nu}(\bm{b})$ be the outward normal at $\bm{b}$, and 
$r = \|\bm{b}-\bm{x}\|_{2}$, we have that
$\bm{x}=\bm{b}-{\rm sgn}(\bm{x}) r \bm{\nu}_{b}$, where 
${\rm sgn}(\bm{x}) = 1,0,-1$ for 
$\bm{x}$ in $\Omega$, $\partial \Omega$, $\Omega^E$, respectively.

For a given target $\bm{x}=(x_{1},x_{2})$ with closest boundary point $\bm{b}$, the
asymptotic analysis is simplest in a coordinate system obtain by 
rotation and translation so that the tangent line
at $\bm{b}$ is aligned with the $x_1$ direction, and translated such that: $\bm{b}$ lies at the origin for layer potentials and $\bm{x}$ lies at the origin for volume potentials. We then reparametrize
the curve $\Gamma$ locally as $x_2 = \gamma(x_1)$ and reparametrize the densities
$\sigma(\bm{x}_\Gamma)$, $\mu(\bm{x}_\Gamma)$ as functions along the tangent line
as well (depicted in \cref{fig-discrete}). 
More precisely, let $\sigma_b = \sigma(\bm{b})$ and $\mu_b = \mu(\bm{b})$, and denoting by $\xi$ the excursion along the tangent
line, we have $x_1(\xi) = \xi$,
\begin{align}
x_2(\xi) &\approx 
\frac12 \gamma^{\prime\prime} \xi^2 + \frac16 \gamma^{(3)} \xi^3 +\frac1{24} \gamma^{(4)} \xi^4 + \dots
\nonumber \\
\sigma(\xi) &\approx 
\sigma_b + \sigma_b^{\prime} \xi + \frac12 \sigma_b^{\prime\prime} \xi^2 + \dots
\label{tanparams} \\
\mu(\xi) &\approx 
\mu_b + \mu_b^{\prime} \xi + \frac12 \mu_b^{\prime\prime} \xi^2 + \dots
\nonumber 
\end{align}
Note that in this parametrization, $\gamma^{\prime\prime}$ is the curvature of the boundary
at $\bm{b}$. This transition is particularly simple given the 
chordal parametrization.

\begin{lemma} \label{lemma:slp}
For $\bm{x} \in \mathbb{R}^2$, let $\bm{b}$ denote the closest point on 
the boundary $\partial\Omega$, let $c = r/\sqrt{\delta}$, and let $\kappa_b$
denote the curvature of the boundary at $\bm{b}$. 
Let $\sigma_b = \sigma(\bm{b})$ and let $\delta = \delta_1$ denote the 
short-time cutoff in the definition \eqref{sheatloc}. Then
\begin{align}\label{eq:slplocO5}
  \begin{split}
  \mathcal{S}_{L}[\sigma](\bm{x}) &= \sqrt{\delta }\frac{\sigma_{b}}{2}\left(c \,\mathrm{erfc}\left(\frac{c}{2}\right)-\frac{2 e^{-\frac{c^2}{4}}}{\sqrt{\pi }}\right)+
{\rm sgn}(\bm{x}) 
\delta \frac{c  \kappa_{b} \sigma_{b} }{4}\left(c \,\mathrm{erfc}\left(\frac{c}{2}\right)-\frac{2 e^{-\frac{c^2}{4}}}{\sqrt{\pi }}\right)\\
   &+\delta ^{3/2}\frac{\sigma_{b} \kappa_{b}^2}{12} \left(2 c^3 \mathrm{erfc}\left(\frac{c}{2}\right)-\frac{\left(4 c^2+1\right) e^{-\frac{c^2}{4}}}{\sqrt{\pi }}\right) +  \delta ^{3/2}\frac{\sigma^{\prime\prime}_{b}}{12}\left(\frac{2 \left(c^2-2\right) e^{-\frac{c^2}{4}}}{\sqrt{\pi }}-c^3 \mathrm{erfc}\left(\frac{c}{2}\right)\right)\\
   &+\delta ^2 {\rm sgn}(\bm{x}) \left[
\frac{c  \kappa_{b}^3 \sigma_{b} }{16} \left(3 c^3 \mathrm{erfc}\left(\frac{c}{2}\right)-\frac{\left(6 c^2-2\right) e^{-\frac{c^2}{4}}}{\sqrt{\pi }}\right) + \frac{c  \kappa_{b} \sigma^{\prime\prime}_{b}}{8} \left(\frac{2 \left(c^2-2\right) e^{-\frac{c^2}{4}}}{\sqrt{\pi }}-c^3 \mathrm{erfc}\left(\frac{c}{2}\right)\right) \right] \\
   & + 
\delta ^2 {\rm sgn}(\bm{x}) \left[ 
\frac{c \sigma_{b} \gamma_{b}^{(4)}}{48}\!\left(\frac{2 \left(c^2-2\right) e^{-\frac{c^2}{4}}}{\sqrt{\pi }}-c^3 \mathrm{erfc}\left(\frac{c}{2}\right)\right) + \frac{c \sigma^{\prime}_{b} \gamma_{b}^{(3)}}{12}\!\left(e^{-\frac{c^2}{4}}\frac{(2 c^2-4)}{\sqrt{\pi}} - c^3 \mathrm{erfc}\left(\frac{c}{2}\right)\right) \right] \\
   &+O\!\left(\delta^{\frac{5}{2}}\right),
\end{split}  
\end{align}  
where $\mathrm{erf}$ denotes the error function and
$\mathrm{erfc}$ denotes the complementary error function.
Here, $\gamma_{b}^{(3)}$, and $\gamma_{b}^{(4)}$, denote the third and fourth derivatives of the curve $\partial\Omega$ at $\bm{b}$ with respect to 
the tangent line parametrization.
$\sigma_b^{\prime}$ and $\sigma_b^{\prime\prime}$ are the first and second derivative of the density, also in the
tangent line parametrization at $\bm{b}$.
\end{lemma}

\begin{proof}
The formulas follow from a straightforward but tedious calculation using 
the change of variables
$z =\sqrt{4t}$, $u = \xi/z$, and expanding all terms in
\eqref{sheatloc} to sufficiently high order.
\end{proof}

Note that $c = r/\sqrt{\delta}$ is of the order $O(1)$ even as 
$\delta \rightarrow 0$ at any location where the asymptotic formula will be 
invoked, because $\calS_L[\sigma]$ has decayed to machine precision
once the target point is about $12 \sqrt{\delta}$ away from the boundary.

\begin{corollary}
In the preceding lemma, when $\bm{x} \in \partial \Omega$, we have the simpler
formula
\begin{equation}\label{eq:slplocO5onsurf}
    \mathcal{S}_{L}[\sigma](\bm{x}) = -\sqrt{\delta}\frac{\sigma_{b}}{\sqrt{\pi}} - \delta^{\frac{3}{2}}\frac{1}{3\sqrt{\pi}}\left(\frac{\kappa_{b}^{2}\sigma_{b}}{4} + \sigma_{b}^{\prime\prime}\right) + O\!\left(\delta^{\frac{5}{2}}\right).
\end{equation}
\end{corollary}

\begin{lemma}
Under the same hypotheses as in \cref{lemma:slp}, 
let $\mu_b = \mu(\bm{b})$, and let $\mu_b^{\prime}$,
$\mu_b^{\prime\prime}$,  $\mu_{b}^{(3)}$, and $\mu_{b}^{(4)}$ denote its first, second, third, and fourth derivatives
in the tangent parametrization at $\bm{b}$.
Then, for $\bm{x} \in \mathbb{R}^2$,
\begin{align}\label{eq:dlplocO5}
  \begin{split}
      \mathcal{D}_{L}[\mu] &(\bm{x}) 
 = {\rm sgn}(\bm{x}) \frac{\mu_{b}}{2} \mathrm{erfc}\left(\frac{c}{2}\right) 
+ \sqrt{\delta }\frac{\kappa_{b} \mu_{b}}{2 \sqrt{\pi }}e^{-\frac{c^2}{4}}
+\delta\, {\rm sgn}(\bm{x}) \left[ 
\frac{3 c \kappa_{b}^2 \mu_{b} }{8 \sqrt{\pi }}e^{-\frac{c^2}{4}} 
+ \frac{c  \mu^{\prime\prime}_{b}}{4}  \left(\frac{2 e^{-\frac{c^2}{4}}}{\sqrt{\pi }}-c\, \mathrm{erfc}\left(\frac{c}{2}\right)\right) \right] \\
&+\delta^{3/2}\frac{5 \left(c^2-2\right) \kappa_{b}^{3} \mu_{b}}{16 \sqrt{\pi }}e^{-\frac{c^2}{4}} + \delta ^{3/2}\frac{\gamma^{(4)}_{b}\mu_{b}}{4 \sqrt{\pi }}e^{-\frac{c^2}{4}}  \\
&+\delta^{3/2}\frac{ \kappa_{b} \mu^{\prime\prime}_{b}}{4}\left(\frac{2 \left(c^2+1\right) e^{-\frac{c^2}{4}}}{\sqrt{\pi }}-c^3 \mathrm{erfc}\left(\frac{c}{2}\right)\right)+ \delta ^{3/2}\frac{\gamma^{(3)}_{b} \mu^{\prime}_{b}}{12} \left(\frac{2e^{-\frac{c^2}{4}} \left(c^2+4\right)}{ \sqrt{\pi }}-c^3 \mathrm{erfc}\left(\frac{c}{2}\right)\right) \\
&+ 
\delta^2 {\rm sgn}(\bm{x})  \left[ 
\frac{35 c  \kappa_{b}^4 \mu_{b} \left(c^2-6\right) }{128 \sqrt{\pi }}e^{-\frac{c^2}{4}}  
+ \frac{5 c\left(\gamma^{(3)}_{b}\right)^{2} \mu_{b} }{12 \sqrt{\pi }} e^{-\frac{c^2}{4}} 
+ \frac{5 c \kappa_{b}\gamma^{(4)}_{b} \mu_{b}}{8 \sqrt{\pi }}e^{-\frac{c^2}{4}}
\right]   \\
&+
\delta^2 {\rm sgn}(\bm{x}) 5c \left[ 
\frac{\kappa_{b}^2 \mu^{\prime\prime}_{b}}{16} \left(\frac{2 \left(c^{2}+1\right) e^{-\frac{c^2}{4}}}{\sqrt{\pi }}-c^{3} \mathrm{erfc}\left(\frac{c}{2}\right)\right)
+
 \frac{\kappa_{b} \gamma^{(3)}_{b} \mu^{\prime}_{b} }{24 }\left(\frac{2 \left(c^2+4\right)e^{-\frac{c^2}{4}}}{\sqrt{\pi }} - c^3  \mathrm{erfc}\left(\frac{c}{2}\right) \right) \right] \\
&+\delta^{2} {\rm sgn}(\bm{x}) 
\frac{c  \mu^{(4)}_{b}}{48}\left(c^3 \mathrm{erfc}\left(\frac{c}{2}\right)-\frac{2 \left(c^2-2\right) e^{-\frac{c^2}{4}}}{\sqrt{\pi }}\right)
+O\!\left(\delta^{\frac{5}{2}}\right).
  \end{split}
\end{align}
\end{lemma}

\begin{proof}
The formulas follow, as for the single layer case, using the change of variables
$z =\sqrt{4t}$, $u = \xi/z$, expanding all terms 
\eqref{dheatloc} to sufficiently high order.
\end{proof}

\begin{corollary}
In the preceding lemma, when $\bm{x} \in \partial \Omega$, we have the simpler
formula
\begin{align}\label{eq:dlplocO5onsurf}
\mathcal{D}_{L}[\mu](\bm{x}) = \sqrt{\delta}\frac{\kappa_{b}\mu_{b}}{2\sqrt{\pi}} + \delta^{\frac{3}{2}}\left(-\frac{5\kappa_{b}^{3}\mu_{b}}{8\sqrt{\pi}} + \frac{\kappa_{b}\mu_{b}^{\prime\prime}}{2\sqrt{\pi}} + \frac{2\mu_{b}^{\prime}\gamma_{b}^{(3)}}{3\sqrt{\pi}} + \frac{\mu_{b}\gamma_{b}^{(4)}}{4\sqrt{\pi}}\right) + O\! \left(\delta^{\frac{5}{2}}\right).
\end{align}
\end{corollary}
Note that 
\[ 
\lim_{r \rightarrow 0^+}
\calD_L[\mu](\bm{b} \pm r \bm{\nu}({\bm{b}})) =
\mathcal{D}_{L}[\mu](\bm{b}) \pm \frac12 \mu(\bm{b}).
\]
That is, the jump relations for the harmonic double layer potential are
manifested in the local part of the double layer heat potential.

For the volume integral, we have the following asymptotic expansion.

\begin{lemma}
For $\bm{x} \in \Omega$, let $\bm{b}$ denote the closest point on 
the boundary $\partial\Omega$, let $c = r/\sqrt{\delta}$, and let $\kappa_b$
denote the curvature of the boundary at $\bm{b}$. 
Rotate and translate the domain $\Omega$ so that the $\xi$-direction
lies along the tangent line at $\bm{b}$ and the $\eta$-direction is the 
inward normal at $\bm{b}$, with $\bm{x}$ lying at the origin. 
In this coordinate system, let the Taylor series for the 
source distribution be given by
\[ 
f(\bm{x}) = f + f_\xi \xi + f_\eta \eta + \frac12 f_{\xi\xi} \xi^2 +
 f_{\xi\eta} \xi \eta + \frac12 f_{\eta\eta} \eta^2 + \dots
\]
and let $\delta = \delta_1$ denote the 
short-time cutoff in the definition \eqref{heatdecompdefs}. Then
\begin{align}
  \begin{split}
    \label{eq:vplocO5}
  \mathcal{V}_{L}[f](\bm{x}) &= \frac{\delta}{4} 
 \left(2 \mathrm{erf}\left(\frac{c}{2}\right)+\frac{2 e^{-\frac{c^2}{4}} c}{\sqrt{\pi }}- c^2 \mathrm{erfc}\left(\frac{c}{2}\right) +2\right) f \\
      &+ \delta ^{3/2}\frac{ \left((4-2 c^2)e^{-\frac{c^2}{4}}+\sqrt{\pi } c^3 \, \mathrm{erfc}\left(\frac{c}{2}\right)\right) \left(\kappa_{b}  f - 2 f_\eta \right)}{12 \sqrt{\pi }}\\
      &+ \frac{\delta^2}{48}\left(\mathrm{erfc}\left(\frac{c}{2}\right) \left(4 c^4 \kappa_{b}  f_{\eta}-3 \left(c^4+4\right) f_{\eta\eta}+\left(c^4-12\right) f_{\xi\xi}-3 c^4 \kappa_{b} ^2 f\right)\right.\\
                                 &\left.+\frac{2 c \left(c^2-2\right) e^{-\frac{c^2}{4}} \left(-4 \kappa_{b}  f_{\eta}+3 f_{\eta\eta}-f_{\xi\xi}+3 \kappa_{b} ^2 f\right)}{\sqrt{\pi }}+24 \left(f_{\xi\xi}+f_{\eta\eta}\right)\right)\\
                                 &+O\!\left(\delta^{\frac{5}{2}}\right).
  \end{split}
\end{align}
\end{lemma}

\begin{proof}
With the change of variables $z = \sqrt{4t}$, $u = \xi/z$, $v= \eta/z$, we can write
the local part of the volume heat potential from
\eqref{heatdecompdefs} as
\begin{equation}\label{eq:vlocuint}
  \mathcal{V}_{L}[f](\bm{x}) = \frac{1}{2\pi}\int_{0}^{\sqrt{4\delta}}\!\int_{-\infty}^{\infty}\int_{\frac{\gamma(uz)}{z}}^{\infty}\! z\,e^{-u^{2}}e^{-v^{2}}f(uz,vz)\,\mathrm{d}v\,\mathrm{d}u\,\mathrm{d}z.
\end{equation}
The asymptotic expansion follows after a modest amount of algebra, using the
relations
\begin{equation}\label{eq:spatmom2d}
  \int_{-\infty}^{\infty} e^{-u^{2}}u^{n}\,\mathrm{d}u = 
\left\{
\begin{array}{ll}
\frac{(2m)!}{4^{m}m!}\sqrt{\pi} & \text{for}\ n=2m,\ m\in \mathbb{N}, \\
& \\
0  & \text{for}\ n=2m+1\, ,
\end{array}
\right.
\end{equation}
and
\begin{equation}\label{eq:myYint}
  \int_{a}^{\infty}\!e^{-u^{2}}u^{n}\,\mathrm{d}u = \frac{1}{2}\Gamma\! \left(\frac{1+n}{2},a^{2}\right),
\end{equation}
where $\Gamma(s,x)$ is the upper incomplete gamma function.
\end{proof}

For target points $\bm{x}$ away from the boundary (at a distance
greater than $12\sqrt{\delta}$), we may ignore exponentially small terms and obtain
the simpler formula
\begin{equation}
    \mathcal{V}_{L}[f](\bm{x}) = \delta f + \frac{\delta^{2}}{2} \left(f_{\xi\xi}+f_{\eta\eta}\right) + O\!\left(\delta^{3}\right).
\end{equation}
For target points $\bm{x}$ {\em on} the boundary $\partial\Omega$, we have 
\begin{equation}
    \mathcal{V}_{L}[f](\bm{x}) = \frac{\delta}{2}  f +  
 \frac{\delta ^{3/2}}{3 \sqrt{\pi }}\left(2 f_{\eta}-\kappa_{b}  f\right)+
\frac{\delta ^2}{4}\left( f_{\xi\xi}+ f_{\eta\eta}\right)+ O\!\left(\delta^{\frac{5}{2}}\right).
  \end{equation}

\begin{remark} \label{rmk:asymp}
Asymptotic expansions of local {\em layer} heat potentials have been 
described in previous papers, including
\cite{gs:fhp,li2009sisc,strainfpt,wang2019acom}. 
Using slightly different language, they 
have also been used
to develop high order quadrature methods based on 
{\em regularized kernels} combined with local correction or Richardson
extrapolation
(see, for example, 
\cite{bealetlupova2024,bealetlupova2024b,beale2016,cortez2001,cortez2005}
and the references therein).
An important feature of all these methods, including our own,
is that they address both the on-surface singular quadrature problem and
the off-surface ``close-to-touching" problem (that is the 
issue of evaluating nearly singular layer potentials) 
using only smooth quadrature rules. 
The uniform treatment of on and off-surface quadrature 
holds true for QBX (quadrature by expansion)
\cite{af2018adaptive,ABarnettquad,klockner_2013b,KLOCKNER2013332,Siegel2018}
as well. QBX involves constructing a local expansion for a layer potential 
centered at a point off-surface, with a radius of convergence which 
reaches back to the boundary itself. 
It involves only smooth quadrature but with 
some degree of oversampling and some ill-conditioning because of the (nearby)
Green's function singularity. We should also note the method of 
\cite{CARVALHO2018327} which, given accurate on-surface values, 
constructs an asymptotic expansion for 
close-to-touching points in terms of distance to the boundary.

Finally, we should observe that in
two dimensions, there are numerous highly effective quadrature methods 
that deal with the Green's function singularity directly, 
which we do not seek to review here (see
\cite{BarnettVeeraWuquad,ggq1,hao2014high,HelsingOjala}). 
In three dimensions, however, our scheme is virtually unchanged, while direct
quadrature methods are much less efficient in terms of the associated constants.

An important difference between the approach presented in this paper 
and earlier
work (excepting \cite{strainfpt}) is that we do not view the task as one of
regularizing the singularity in the harmonic Green's function ( $\log r$ in
2D and $1/r$ in 3D) or one of integrating it against a smooth density.
Instead, we have fully switched to the heat potential
formulation, involving only $C^\infty$ functions and the NUFFT 
for the history part coupled with
asymptotic expansions for the local part. Those expansions are robust and stable
independent of distance from the boundary.
Moreover, in contrast to earlier schemes, we can easily achieve
any desired accuracy by combining our asymptotic formulas with a telescoping, 
local quadrature, described briefly in the next section.
\end{remark}
\subsection{Computing the local part of volume and layer potentials with 
 arbitrary precision} \label{s:dyadicref}
The scheme described in \cref{s:local} can achieve convergence rates
of up to fifth order accuracy, assuming $\delta_1 = O(\Delta x^2)$ 
and using asymptotic expansions of order $\delta_1^p$ with $p \leq {5/2}$.
(We typically 
choose to truncated the expansions at order $\delta_1^{3/2}$ to avoid 
the need for high order derivatives of the boundary, as discussed in the 
examples below.)
Suppose, however, that higher precision is desired in the local part. 
Reducing $\delta_1$, unfortunately,
would cause the size of the Fourier transforms to increase, making the 
near history contribution extremely expensive. 
Instead, we combine two ideas from the numerical
analysis of heat potentials: full product integration in time from \cite{li2009sisc}
and the introduction of a second small parameter $\delta^\ast$
for which the asymptotic contribution is guaranteed to have the requested accuracy 
from \cite{wang2019acom}. We illustrate the idea using the single layer potential. 

Suppose that one wishes
to compute $\calS_L[\sigma]$ to some precision $\epsilon < \delta_1^p$.
For this, first 
determine the smallest $J$ such that $[\delta_1/(4^J)]^p \leq \epsilon$,
so that an asymptotic expansion with parameter
$\delta^\ast = \delta_1/(4^J)$ is sufficiently accurate.
If $J=0$, then the local part is already being computed with the desired accuracy. 

Otherwise, following \cite{li2009sisc,wang2019acom}, write
\begin{equation}
\begin{aligned}
\mathcal{S}^\ast_L[\sigma](\bm{x}) &=  
\int_0^{\delta^\ast}\! \int_{\partial \Omega}
\frac{e^{-\|\bm{x-x'}\|^2/4t}}{4 \pi t} \, \sigma(\bm{x'}) \,\mathrm{ds}_{\bm{x'}} \, \mathrm{d}t
+ \sum_{j=1}^J
\int_{\delta_1/(4^j)}^{\delta_1/(4^{j-1})}\! \int_{\partial \Omega}
\frac{e^{-\|\bm{x-x'}\|^2/4t}}{4 \pi t} \, \sigma(\bm{x'}) \,\mathrm{ds}_{\bm{x'}} \, \mathrm{d}t \\
&=  
\int_0^{\delta^\ast}\! \int_{\partial \Omega}
\frac{e^{-\|\bm{x-x'}\|^2/4t}}{4 \pi t} \, \sigma(\bm{x'}) \,\mathrm{ds}_{\bm{x'}} \, \mathrm{d}t
+ 
\sum_{j=1}^J
\int_{\partial \Omega} 
{\cal K}_j(\|\bm{x-x'}\|^2/\delta_1)
\, \sigma(\bm{x'}) \,\mathrm{ds}_{\bm{x'}} \, ,
\end{aligned}
\end{equation}
where the {\em difference kernel} ${\cal K}_j(r)$ is given by
\[ {\cal K}_j(r) =  \frac{1}{4\pi}
\left[ E_1( 4^{j-2} r) - 
E_1( 4^{j-1} r) \right]  \, 
\]
from \eqref{eiformula}.
The first term (the integral over the interval $[0,\delta^\ast]$)
is evaluated to the desired precision using asymptotics.
Each of the terms involving difference kernels
is infinitely differentiable but 
more sharply peaked with 
each refinement level. In fact, it is easy to verify that
the range of the kernel to any fixed precision decreases by a factor of two
when $j \rightarrow j+1$, requiring a doubling of 
the number of quadrature points over the range of interest.
Taking the perspective of a target point $\bm{x'}$, however,
the range of non-negligible
interactions has shrunk by a factor of two at the same time. 
Combining these facts, it is easy to see that the 
number of quadrature nodes that make a non-negligible contribution to a target
remains the same and thus, the work for each target is 
constant at every level using only a smooth quadrature rule
and direct summation (without the NUFFT or any other acceleration).
This idea extends naturally to the double layer and volume heat 
potential (and is used in a slightly different form 
in the DMK-based box code of \cite{dmk}).

Since $\sqrt{\delta_1} \approx \Delta x$, the
telescoping series correction is entirely local for any chunk, and the 
only aspect of the procedure that requires some care is that each
chunk should be processed independently in order to avoid excessive storage.
We omit further details in the present paper for the sake of simplicity.
In our third order accurate experiments on quasi-uniform triangulations,
we determined that the error constants stemming from layer potential
asymptotics generally exceed those from volume potential asymptotics.
Thus, we apply this telescoping correction only to the local layer 
potentials.
Carrying out three steps of 
dyadic refinement for $\calS_L[\sigma]$ and $\calD_L[\mu]$ clearly reduces the 
associate error constant in the 
layer potential part
by a factor of $(1/64)^p$, which is more than sufficient in all of our tests.
Detailed timings will be presented in \cref{s:numericalresults}.
\section{The boundary integral correction} \label{s:bie}
When solving an interior boundary value problem in a simply
connected domain with Dirichlet conditions,
standard potential theory \cite{guenther1988partial,kelloggpot,mikhlin,stakgold} 
suggests a representation for the solution $u$ of the form \eqref{dir_rep},
repeated here:
\[
u(\bm{x}) = \mathcal{V}[f](\bm{x}) + \mathcal{D}[\mu](\bm{x})
\quad \text{for }\bm{x}\in\Omega .
\]
Taking the limit as $\bm{x}$ approaches the boundary and using the jump relations
\eqref{dlpjump}, we obtain the integral equation
\begin{equation}
-\frac{1}{2} \mu(\bm{x'}) + \mathcal{D}[\mu](\bm{x'}) =  g(\bm{x'}) - \mathcal{V}[f](\bm{x'})
\quad \text{for }\bm{x'} \in \partial\Omega.
\label{dir_inteq}
\end{equation}

When $\Omega$ is an exterior domain, in addition to the Dirichlet
boundary conditions, the behavior of the solution at infinity
must be specified \cite{GREENBAUM1993267,mikhlin,stakgold}.
More precisely, the coefficient $A$ must be given corresponding
to growth at infinity of the order  
$A \log \| \bm{x} \|$. (Bounded solutions correspond to setting $A=0$.)
A standard representation for the solution is
\begin{equation}
    u(\bm{x}) = \mathcal{V}[f](\bm{x}) + \mathcal{D}_E[\mu](\bm{x})
+ A \log \| \bm{x} - \bm{x}_0 \|
\quad \text{for }\bm{x}\in\Omega,
\label{dir_repext}
\end{equation}
where $\bm{x}_0$ is some point in the interior
of the closed curve $\partial\Omega$, and 
\[
 \mathcal{D}_E[\mu](\bm{x}) \equiv 
\int_{\partial\Omega}
\left[ \frac{\partial G}{\partial {\nu}_{\bm{x'}}}
(\bm{x}-\bm{x'})+ 1 \right] 
\mu(\bm{x'})\,\mathrm{ds}_{\bm{x'}}
\quad \text{for }\bm{x}\in\Omega.
\]

For the Neumann problem, with boundary condition \eqref{neu_rep},
Green's representation formula 
\eqref{neu_rep} leads to the integral equation
\begin{equation}
 \frac12 u(\bm{x'}) + \mathcal{D}[u](\bm{x'})
= \mathcal{V}[f](\bm{x'}) + 
\mathcal{S}[g] (\bm{x'})
\quad \text{for }\bm{x'} \in \partial\Omega.
\end{equation}

These boundary integral equations are classical
\cite{atkinson_1997,KressRainer2014LIE} and there are numerous
fast solvers that have been developed for their solution, described for example
in the text \cite{martinsson2019fast}.
In the present paper, we make use of GMRES iteration coupled with
fast multipole acceleration \cite{fmm2dlib}, as in 
\cite{GREENBAUM1993267}. 
Since this is a mature subject, and we are using standard techniques, we omit further details.
\section{Numerical results}\label{s:numericalresults}
We have implemented the algorithms above in Fortran 77/90, using the {\tt fmm2d} library 
\cite{fmm2dlib} combined with GMRES iteration to solve the relevant
boundary integral equations and the {\tt FINUFFT} library \cite{finufftlib} to 
compute the nonuniform fast Fourier transforms. The timings below were obtained 
using an AMD Ryzen 7 PRO 6850U CPU with 16 Mb cache, with everything
compiled in single-threaded mode so that we can easily compute the number of points
processed per second per core (a useful figure of merit for linear or nearly linear scaling algorithms).

Given a description of the boundary $\Gamma$, we generate a curvilinear
triangulation using {\tt MeshPy} \cite{meshpylib} as an interface to 
{\tt Triangle} \cite{trianglepaper,trianglelib}. 
In the interior of each triangular element, we use six Vioreanu-Rokhlin nodes \cite{vioreanu2014sisc}, obtained from \cite{modepylib}. This is 
sufficient for integrating polynomials of
degree two, resulting in third order accuracy. 
In all experiments, we generate a sequence of up to seven meshes, 
with 
\[ h\in\{0.1, 0.075, 0.05, 0.025, 0.01, 0.0075, 0.005\}, \] 
where $h$ is the length of the longest side amongst all mesh elements. 
Given such a triangulation with quadrature weights $w_i,\ i=1,\dots,N$, 
we define 
$\Delta x^2 = \frac{1}{N}\sum_{i=1}^{N} w_{i}$
and set $\delta = 3 \Delta x^2$.

In our tests, all boundary curves are assumed to take the form 
\begin{equation}\label{eq:bdryparam}
  \bm{\gamma}(t) = r\left(\sum\limits_{k=0}^{N_{c}}\alpha_{k}\cos(kt) + \sum\limits_{k=1}^{N_{s}}\beta_{k}\sin(kt)\right) \begin{bmatrix}\cos(st)\\\sin(st)\end{bmatrix} + \begin{bmatrix}x_{0}\\y_{0}\end{bmatrix}\quad {\rm for}\ t\in[0,2\pi),
\end{equation}
parametrized by $r,\{ \alpha_{k} \},\{ \beta_{k} \},x_{0},y_{0},s$.
We set the tolerance for NUFFT to $10^{-6}$, 
the tolerance for GMRES to $10^{-8}$, and the tolerance for the FMM to $10^{-6}$.
The tolerance $\epsilon$ for the history part \eqref{historytol} 
is also set to $10^{-6}$.

Errors are computed over all quadrature nodes in the 
volumetric discretization  in $\ell_{\infty}$ and $\ell_{2}$ norms.
Since, for interior Neumann problems, the solution is only unique up to an arbitrary 
constant, we define the error as
\begin{equation}
  e_{\infty} = \frac{\| (\tilde{u}-\tilde{u}_{c}) -(u-u_{c}) \|_{\infty}}{\|u-u_{c}\|_{\infty}},\quad   e_{2} = \frac{\|(\tilde{u}-\tilde{u}_{c})-(u-u_{c})\|_{2}}{\|u-u_{c}\|_{2}},
\end{equation}
where $u_{c} = u(\bm{x}_{c})$ for a randomly chosen point $\bm{x}_{c}$
and $\tilde{u}$ is our approximate computed solution.
For Dirichlet boundary conditions, we may set $\tilde{u}_c = u_{c}=0$.

In our first experiment, we verify the convergence rates for our 
asymptotic expansions.
Our second example is an interior Neumann problem and our third example 
is an exterior Dirichlet problem. For our last example, we study the
effect of poor mesh quality (degenerate triangles, high aspect
ratio elements and gaps) in the solution of an interior Dirichlet problem.
\subsection{Validation of high order asymptotic expansions}\label{ss:exvalidation}

For the single and double layer potentials, we consider the 
domain $\Omega$ with boundary given by \cref{eq:bdryparam}, 
with nonzero parameters $r=1$, $\alpha_{0}=5/6$, $\beta_{3} = 1/30$, $\alpha_{5} = 1/15$, $\alpha_{6} = 1/30$, $\alpha_{8} = 1/30$, $\alpha_{10} = 1/30$, and $s=1$.
(See \cref{fig:exintneu_omega_and_sol}.)
We assume the layer potential densities are given by
\[
  \sigma(t) = \mu(t) = 1 + \sin(2t) + \cos(4t) \quad \text{for } t\in \halfopen{0}{2\pi}. 
\]

We measure the error in evaluating 
$\mathcal{S}_L[\sigma]$ and $\mathcal{D}_L[\mu]$ at $10000$ random locations in the box $[-1,1]^{2}$, which entirely contains $\bar{\Omega}$. All curve and 
density derivatives that appear in the asymptotic formulas are computed 
analytically. We also use a sixteenth order discretization of the boundary 
$\partial\Omega$, and set all tolerances to fourteen digits of accuracy 
in order to ensure that errors are due exclusively to the asymptotic 
approximations. The results are shown in the left and center panels
of \cref{fig:exval_conv_vol}, where we observe the expected order of convergence from
 \cref{eq:slplocO5} and \cref{eq:dlplocO5}.

\begin{figure}
\centering
\subfloat{
%
%
\definecolor{mycolor1}{rgb}{0.00000,0.44700,0.74100}%
\definecolor{mycolor2}{rgb}{0.85000,0.32500,0.09800}%
\begin{tikzpicture}[scale=0.55]

\begin{axis}[%
scale only axis,
xmode=log,
xmin=0.000001,
xmax=0.003,
xminorticks=true,
xlabel style={font=\Large\color{white!15!black}},
xlabel={$\delta$},
xticklabel style = {font=\large},
extra x ticks={0.002},
extra x tick labels={},
ymode=log,
ymin=1e-12,
ymax=0.001,
yminorticks=true,
ylabel style={font=\Large\color{white!15!black}},
ylabel={Relative error},
yticklabel style = {font=\large},
axis background/.style={fill=white},
xmajorgrids,
xminorgrids,
ymajorgrids,
yminorgrids,
legend style={font=\Large,at={(0.03,0.97)}, anchor=north west, legend cell align=left, align=left, draw=white!15!black},
clip mode=individual,
]
\addplot [color=mycolor1, line width=1.0pt, mark=square*, mark options={solid, mycolor1}, mark size=2.0pt]
  table[row sep=crcr]{%
1.000000e-03 6.965999e-04 \\
5.274997e-04 1.820020e-04 \\
2.782559e-04 5.289530e-05 \\
1.467799e-04 1.325988e-05 \\
7.742637e-05 2.732296e-06 \\
4.084239e-05 8.285891e-07 \\
2.154435e-05 2.349053e-07 \\
1.136464e-05 5.610848e-08 \\
5.994843e-06 1.137074e-08 \\
3.162278e-06 3.068346e-09 \\
};

\addplot [color=mycolor2, dotted, line width=1.0pt, mark=square, mark options={solid, mycolor2}, mark size=2.0pt]
  table[row sep=crcr]{%
1.000000e-03 8.691356e-05 \\
5.274997e-04 2.011701e-05 \\
2.782559e-04 4.696373e-06 \\
1.467799e-04 1.106289e-06 \\
7.742637e-05 2.567208e-07 \\
4.084239e-05 6.054419e-08 \\
2.154435e-05 1.434273e-08 \\
1.136464e-05 3.511387e-09 \\
5.994843e-06 8.152016e-10 \\
3.162278e-06 1.663777e-10 \\
};

\addplot [color=black, dashed, line width=1.0pt]
  table[row sep=crcr]{%
7.742637e-05 5.518275e-06 \\
4.084239e-05 1.535493e-06 \\
2.154435e-05 4.272600e-07 \\
1.136464e-05 1.188876e-07 \\
5.994843e-06 3.308119e-08 \\
};

\addplot [color=mycolor1, line width=1.0pt, mark=*, mark options={solid, mycolor1}, mark size=2.0pt]
  table[row sep=crcr]{%
1.000000e-03 4.717426e-04 \\
5.274997e-04 1.158198e-04 \\
2.782559e-04 2.247010e-05 \\
1.467799e-04 5.795497e-06 \\
7.742637e-05 1.223693e-06 \\
4.084239e-05 1.939333e-07 \\
2.154435e-05 4.498951e-08 \\
1.136464e-05 1.003251e-08 \\
5.994843e-06 1.962563e-09 \\
3.162278e-06 2.810016e-10 \\
};

\addplot [color=mycolor2, dotted, line width=1.0pt, mark=o, mark options={solid, mycolor2}, mark size=2.0pt]
  table[row sep=crcr]{%
1.000000e-03 5.478049e-05 \\
5.274997e-04 1.086689e-05 \\
2.782559e-04 2.014275e-06 \\
1.467799e-04 3.671940e-07 \\
7.742637e-05 6.877565e-08 \\
4.084239e-05 1.228032e-08 \\
2.154435e-05 2.136913e-09 \\
1.136464e-05 3.882554e-10 \\
5.994843e-06 6.840687e-11 \\
3.162278e-06 1.012669e-11 \\
};

\addplot [color=black, dashdotdotted, line width=1.0pt]
  table[row sep=crcr]{%
7.742637e-05 1.201572e-07 \\
4.084239e-05 2.428317e-08 \\
2.154435e-05 4.907507e-09 \\
1.136464e-05 9.917826e-10 \\
5.994843e-06 2.004343e-10 \\
};
\end{axis}

\end{tikzpicture}%
}
\subfloat{
%
%
\definecolor{mycolor1}{rgb}{0.00000,0.44700,0.74100}%
\definecolor{mycolor2}{rgb}{0.85000,0.32500,0.09800}%
\begin{tikzpicture}[scale=0.55]

\begin{axis}[%
scale only axis,
xmode=log,
xmin=0.000001,
xmax=0.003,
xminorticks=true,
xlabel style={font=\Large\color{white!15!black}},
xlabel={$\delta$},
xticklabel style = {font=\large},
extra x ticks={0.002},
extra x tick labels={},
ymode=log,
ymin=1e-10,
ymax=0.01,
yminorticks=true,
ylabel style={font=\Large\color{white!15!black}},
ylabel={Relative error},
yticklabel style = {font=\large},
axis background/.style={fill=white},
xmajorgrids,
xminorgrids,
ymajorgrids,
yminorgrids,
legend style={font=\Large,at={(0.03,0.97)}, anchor=north west, legend cell align=left, align=left, draw=white!15!black},
clip mode=individual,
]
\addplot [color=mycolor1, line width=1.0pt, mark=square*, mark options={solid, mycolor1}, mark size=2.0pt]
  table[row sep=crcr]{%
1.000000e-03 1.167069e-02 \\ 
5.274997e-04 3.423291e-03 \\ 
2.782559e-04 8.917484e-04 \\ 
1.467799e-04 1.939171e-04 \\ 
7.742637e-05 7.029816e-05 \\ 
4.084239e-05 1.748774e-05 \\ 
2.154435e-05 4.062158e-06 \\ 
1.136464e-05 1.517986e-06 \\ 
5.994843e-06 4.576037e-07 \\ 
3.162278e-06 1.000210e-07 \\ 
};

\addplot [color=mycolor2, dotted, line width=1.0pt, mark=square, mark options={solid, mycolor2}, mark size=2.0pt]
  table[row sep=crcr]{%
1.000000e-03 1.045028e-03 \\ 
5.274997e-04 2.663575e-04 \\ 
2.782559e-04 6.385036e-05 \\ 
1.467799e-04 1.496642e-05 \\ 
7.742637e-05 3.819651e-06 \\ 
4.084239e-05 9.642586e-07 \\ 
2.154435e-05 2.337636e-07 \\ 
1.136464e-05 5.984352e-08 \\ 
5.994843e-06 1.495002e-08 \\ 
3.162278e-06 3.186447e-09 \\ 
};

\addplot [color=black, dashed, line width=1.0pt]
  table[row sep=crcr]{%
7.742637e-05 1.199221e-04 \\ 
4.084239e-05 3.336902e-05 \\ 
2.154435e-05 9.285129e-06 \\ 
1.136464e-05 2.583642e-06 \\ 
5.994843e-06 7.189138e-07 \\ 
};
\addplot [color=mycolor1, line width=1.0pt, mark=*, mark options={solid, mycolor1}, mark size=2.0pt]
  table[row sep=crcr]{%
 1.000000e-03 1.224455e-02 \\ 
5.274997e-04 3.027073e-03 \\ 
2.782559e-04 6.410129e-04 \\ 
1.467799e-04 1.431911e-04 \\ 
7.742637e-05 2.968500e-05 \\ 
4.084239e-05 5.639604e-06 \\ 
2.154435e-05 9.383464e-07 \\ 
1.136464e-05 1.179222e-07 \\ 
5.994843e-06 1.465262e-08 \\ 
3.162278e-06 2.722173e-09 \\ 
};

\addplot [color=mycolor2, dotted, line width=1.0pt, mark=o, mark options={solid, mycolor2}, mark size=2.0pt]
  table[row sep=crcr]{%
1.000000e-03 9.060404e-04 \\ 
5.274997e-04 1.867084e-04 \\ 
2.782559e-04 3.828869e-05 \\ 
1.467799e-04 7.214800e-06 \\ 
7.742637e-05 1.222515e-06 \\ 
4.084239e-05 2.122962e-07 \\ 
2.154435e-05 3.691164e-08 \\ 
1.136464e-05 5.502208e-09 \\ 
5.994843e-06 7.691716e-10 \\ 
3.162278e-06 1.422768e-10 \\ 
};

\addplot [color=black, dashdotdotted, line width=1.0pt]
  table[row sep=crcr]{%
7.742637e-05 2.110213e-07 \\ 
4.084239e-05 4.264634e-08 \\ 
2.154435e-05 8.618613e-09 \\ 
1.136464e-05 1.741778e-09 \\ 
5.994843e-06 3.520047e-10 \\ 
};
\end{axis}

\end{tikzpicture}%
}
\subfloat{
%
%
\definecolor{mycolor1}{rgb}{0.00000,0.44700,0.74100}%
\definecolor{mycolor2}{rgb}{0.85000,0.32500,0.09800}%
\begin{tikzpicture}[scale=0.55]

\begin{axis}[%
scale only axis,
xmode=log,
xmin=0.000001,
xmax=0.003,
xminorticks=true,
xlabel style={font=\Large\color{white!15!black}},
xlabel={$\delta$},
xticklabel style = {font=\large},
extra x ticks={0.002},
extra x tick labels={},
ymode=log,
ymin=1e-09,
ymax=1,
yminorticks=true,
ylabel style={font=\Large\color{white!15!black}},
ylabel={Relative error},
yticklabel style = {font=\large},
axis background/.style={fill=white},
xmajorgrids,
xminorgrids,
ymajorgrids,
yminorgrids,
legend style={font=\Large,at={(0.03,0.97)}, anchor=north west, legend cell align=left, align=left, draw=white!15!black},
clip mode=individual,
]
\addplot [color=mycolor1, line width=1.0pt, mark=square*, mark options={solid, mycolor1}, mark size=2.0pt]
  table[row sep=crcr]{%
   0.000961644406940   0.052418548681328\\
   0.000541436243713   0.022276932315192\\
   0.000237951818000   0.005850522940774\\
   0.000060815169185   0.000480222684418\\
   0.000009763223090   0.000013409773430\\
   0.000005494895537   0.000004304253251\\
   0.000002443885747   0.000000857792870\\
};

\addplot [color=mycolor2, dotted, line width=1.0pt, mark=square, mark options={solid, mycolor2}, mark size=2.0pt]
  table[row sep=crcr]{%
   0.000961644406940   0.011835289670045\\
   0.000541436243713   0.005030375486929\\
   0.000237951818000   0.001294998964960\\
   0.000060815169185   0.000105142137848\\
   0.000009763223090   0.000002917722176\\
   0.000005494895537   0.000000930302648\\
   0.000002443885747   0.000000184887199\\
};

\addplot [color=black, dashed, line width=1.0pt]
  table[row sep=crcr]{%
   0.000237951818000   0.018379635559094\\
   0.000060815169185   0.001200556711046\\
   0.000009763223090   0.000030941778108\\
};

\addplot [color=mycolor1, line width=1.0pt, mark=*, mark options={solid, mycolor1}, mark size=2.0pt]
  table[row sep=crcr]{%
   0.000961644406940   0.080967361355944\\
   0.000541436243713   0.019876793653893\\
   0.000237951818000   0.002274445795928\\
   0.000060815169185   0.000048753018903\\
   0.000009763223090   0.000000331566570\\
   0.000005494895537   0.000000088167434\\
   0.000002443885747   0.000000028095591\\
};

\addplot [color=mycolor2, dotted, line width=1.0pt, mark=o, mark options={solid, mycolor2}, mark size=2.0pt]
  table[row sep=crcr]{%
   0.000961644406940   0.017558435466307\\
   0.000541436243713   0.004308001782118\\
   0.000237951818000   0.000498894356057\\
   0.000060815169185   0.000010624205802\\
   0.000009763223090   0.000000056767682\\
   0.000005494895537   0.000000015391968\\
   0.000002443885747   0.000000003995063\\
};

\addplot [color=black, dashdotdotted, line width=1.0pt]
  table[row sep=crcr]{%
   0.000237951818000   0.00003690917110281\\
   0.000060815169185   0.00000121882547258\\
   0.000009763223090   0.00000001258620371\\
};

\end{axis}

\end{tikzpicture}%
}
\caption{Relative errors as functions of $\delta$ for 
$\mathcal{S}_L[\sigma]$ (left), $\mathcal{D}_L[\mu]$ (center) and
$\mathcal{V}_L[f]$ (right) using asymptotic expansions of order 
$O\!\left(\delta^{p/2}\right)$ with $p = 4,5$.
Solid lines with solid markers are the relative $\ell_{\infty}$ errors, dotted lines with open markers are the relative $\ell_{2}$ errors. Square markers 
($\mathsmaller{\blacksquare}$, $\mathsmaller{\square}$) correspond to $p=4$, 
while round markers ($\bullet$, $\circ$) correspond to $p=5$. 
The dashed and dot-dash lines, without markers, are reference curves
with convergence rate $\delta^{p/2}$ for $p=4, 5$, respectively.}
\label{fig:exval_conv_vol}
\end{figure}
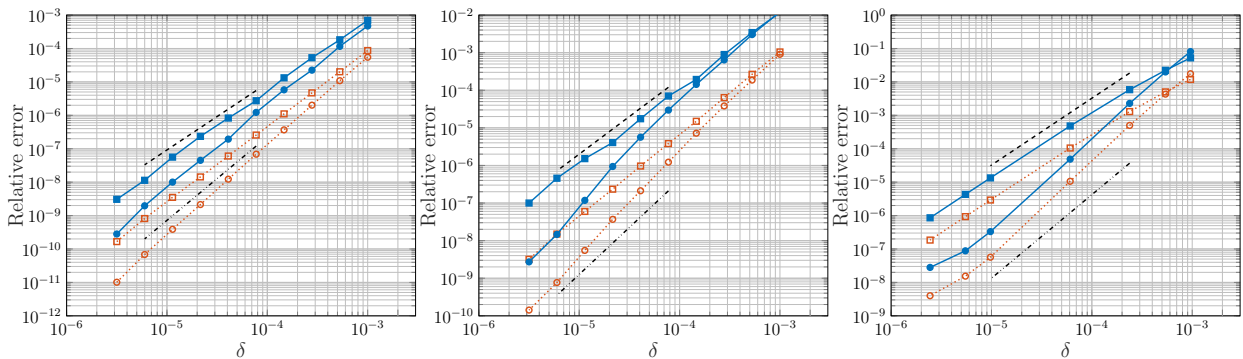

For the volume potential, we solve the interior Dirichlet problem in the same
domain $\Omega$ with volumetric data
\begin{equation}
  f(\bm{x}) = -50\sin(5(x_{1}+x_{2}))+2/9+1000(1000x_{1}^{2}-1)e^{-500x_{1}^{2}}
\quad \text{for }\bm{x}\in\Omega.
\end{equation}

The volume potential $\mathcal{V}[f](\bm{x})$ is computed using the 
near and far history and $\mathcal{V}_{L}[f](\bm{x})$ on a sequence of seven
meshes as indicated above.
In order to make the double layer potential error negligible, we use the 
technique described in \cite{HelsingOjala} with sixteenth order accuracy.
The errors are computed at $5000$ random points in $\Omega$ and 
the results plotted in \cref{fig:exval_conv_vol} (right panel), where we observe the expected orders of convergence from \cref{eq:vplocO5} in the errors as functions of $\delta$.

\subsection{An interior Neumann problem}\label{ss:exintneu}

For the same domain, we consider the exact solution
\begin{equation}\label{eq:neusolana}
  u(\bm{x}) =  \frac{(x_{1}-a)}{(x_{1}-a)^{2} + (x_{2}-b)^{2}} +\frac{x_{1}^{2}}{9}-x_{2} + 8 + e^{-5x_{1}^2} + u_{c} \quad \text{for }\bm{x}\in\Omega,
\end{equation}
where $u_{c}$ is an additive constant, $a = 1.1$, and $b = 1.3$. 

The function $u(\bm{x})$ satisfies the Poisson equation with source density
\begin{equation}\label{eq:exintneusourcedens}
 f(\bm{x}) =\frac{2}{9} + (100x_{1}^{2}-10)e^{-5 x_{1}^2}
\quad \text{for }\bm{x}\in\Omega,
\end{equation}
and we may compute the Neumann boundary data 
$g(\bm{x})=\bm{\nu}(\bm{x})\cdot\nabla u(\bm{x})$ analytically.

In \cref{fig:exintneu_omega_and_sol}, we show a quasi-uniform mesh on $\Omega$ 
for $h\approx 0.1$ together with a plot of the solution $u$, 
given by \cref{eq:neusolana} with $u_{c} = 0$.
\begin{figure}[h]
\centering 
\subfloat{\includegraphics[scale=0.27]{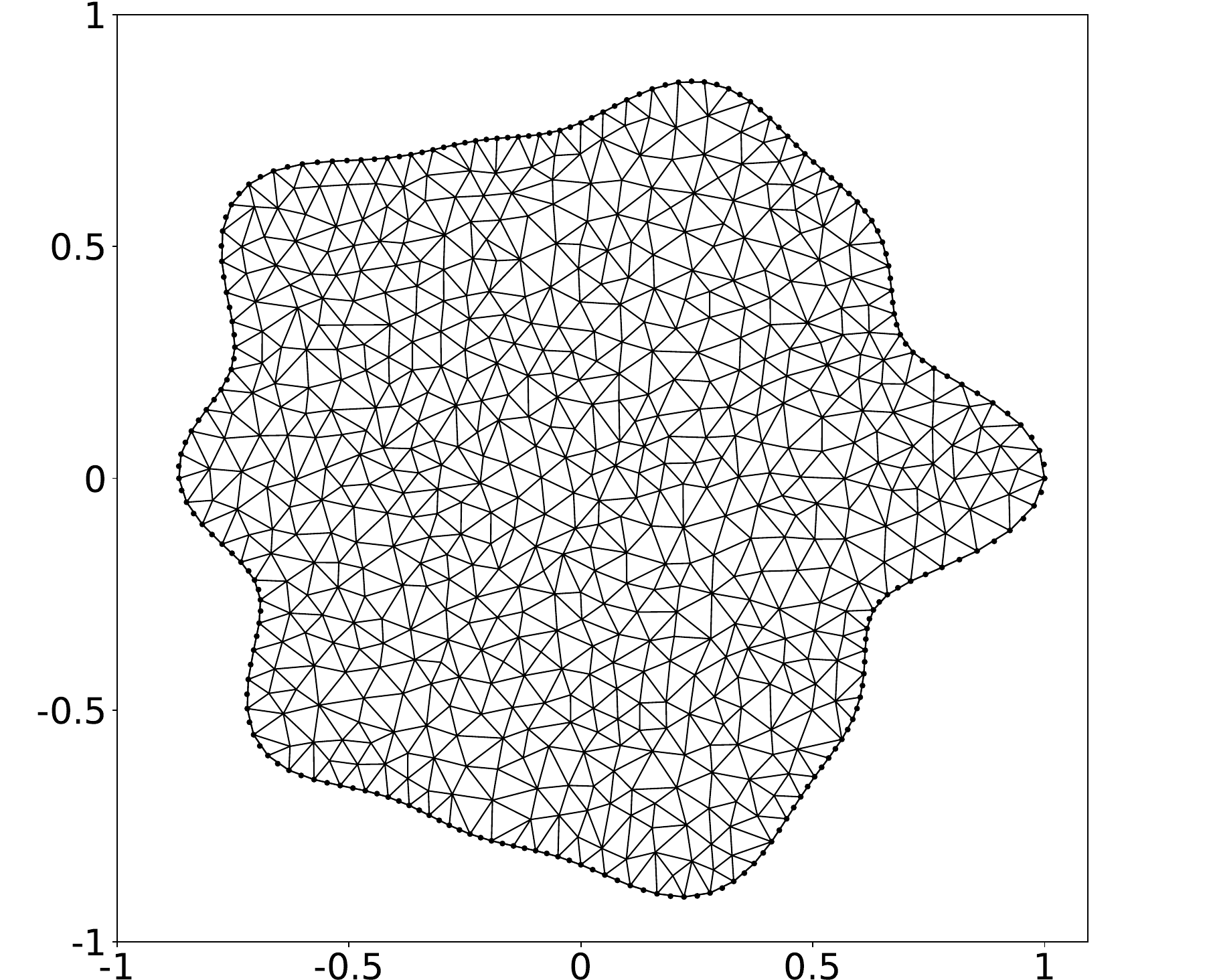}}
\subfloat{\includegraphics[trim={4.5cm 8.8cm 5cm 8cm},clip,scale=0.65]{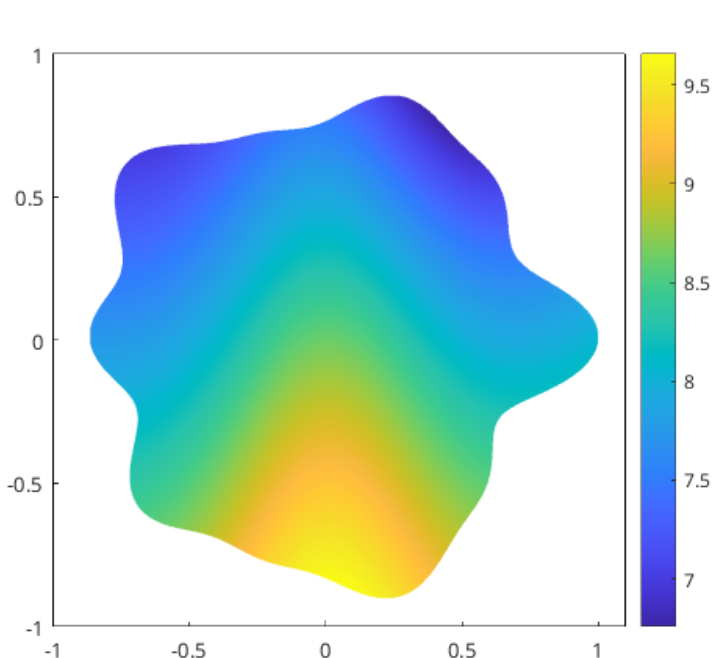}}
\caption{A volumetric mesh for the domain $\Omega$ 
in \cref{ss:exintneu} (left).
There are $1142$ triangles in the discretization, with $h\approx 0.1$ and
$208$ nodes along the boundary $\partial\Omega$, marked with black dots. 
The solution to the interior Neumann problem \cref{eq:neusolana}, 
with $u_{c} = 0$ is plotted on the right.}
\label{fig:exintneu_omega_and_sol}
\end{figure}

We solve the interior Neumann problem for a sequence of seven meshes,
using three levels of dyadic refinement to improve the accuracy of the local
parts of the layer potentials, as described in \cref{s:dyadicref}.
In the asymptotic formulas, we include terms up to $O(\delta)$. The
resulting errors are of the order $O\!\left(\delta^{3/2}\right)$, 
which is equivalent to third order accuracy with respect to $\Delta x$.

The errors in the approximate solutions are plotted 
in \cref{fig:exintneu_conv_and_time} (left panel) as functions of 
$\Delta x$. The convergence rate is clearly third order, until about 
five digits of accuracy are obtained. (The reader may note that
the prescribed tolerances were set to six digits of accuracy. 
We lose one digit of accuracy due to the indeterminacy of the Neumann
problem. Our computed solution happened to be ten times larger than the exact 
solution with $u_c = 0$.)
In \cref{fig:exintneu_conv_and_time} (right panel) we plot the total time for solving and evaluating the solution at all 
interior nodes. These timings do not include mesh generation, but include all 
other steps for solving the Poisson equation. For 
$N_{V} = 674664$, $1198728$, $2695248$ (the three rightmost data points), 
the throughput is $8.3\times 10^{5}$,  $8.7\times 10^{5}$, and 
$8.4\times 10^{5}$ points per second per core.
For problem sizes in this range,
about 80\% of the time is spent in evaluating 
$\mathcal{V}_{NH}[f]$,$\mathcal{V}_{FH}[f],
\mathcal{D}_{NH}[\mu],$ and $\mathcal{D}_{FH}[\mu]$. 
About 10\% of the time is spent in the dyadic refinement
scheme of \cref{s:dyadicref} and another 10\% of the time is spent
solving the boundary
integral equation. The cost of evaluating the asymptotic expansions is negligible.

\begin{figure}[h]
\centering 
\subfloat{
%
%
\definecolor{mycolor1}{rgb}{0.00000,0.44700,0.74100}%
\definecolor{mycolor2}{rgb}{0.85000,0.32500,0.09800}%
\begin{tikzpicture}[scale=0.78]

\begin{axis}[%
scale only axis,
xmode=log,
xmin=0.0008,
xmax=0.03,
xminorticks=true,
xlabel style={font=\Large\color{white!15!black}},
xlabel={$\Delta x$},
xticklabel style = {font=\large},
extra x ticks={0.02},
extra x tick labels={},
ymode=log,
ymin=1e-05,
ymax=0.01,
yminorticks=true,
ylabel style={font=\Large\color{white!15!black}},
ylabel={Relative error},
yticklabel style = {font=\large},
axis background/.style={fill=white},
xmajorgrids,
xminorgrids,
ymajorgrids,
yminorgrids,
legend style={font=\Large,at={(0.03,0.97)}, anchor=north west, legend cell align=left, align=left, draw=white!15!black},
clip mode=individual,
]
\addplot [color=mycolor1, line width=1.0pt, mark=*, mark options={solid, mycolor1}, mark size=2.0pt]
  table[row sep=crcr]{%
0.0179038581218315	0.00168468687322798\\
0.0134342378981593	0.000786364629800699\\
0.00890602451526748	0.000261469304831535\\
0.00450241302656925	3.39539657741937e-05\\
0.00180399769861289	1.61195279037501e-05\\
0.00135337793900323	1.42419147251461e-05\\
0.000902567771538386	1.35483449505379e-05\\
};
\addlegendentry{$e_{\infty}$}

\addplot [color=mycolor2, dotted, line width=1.0pt, mark=o, mark options={solid, mycolor2}, mark size=2.0pt]
  table[row sep=crcr]{%
0.0179038581218315	0.00220106696893435\\
0.0134342378981593	0.00108243317642174\\
0.00890602451526748	0.000336406908541728\\
0.00450241302656925	3.2562369474738e-05\\
0.00180399769861289	1.47609590051477e-05\\
0.00135337793900323	1.2861557674808e-05\\
0.000902567771538386	1.23606760045937e-05\\
};
\addlegendentry{$e_{2}$}

\addplot [color=black, dashed, line width=1.0pt]
  table[row sep=crcr]{%
0.0179038581218315	0.00426996571900406\\
0.0134342378981593	0.0018039463170365\\
0.00890602451526748	0.000525576773152868\\
0.00450241302656925	6.79079315483874e-05\\
};
\addlegendentry{$O(\Delta x ^{3})$}

\end{axis}

\end{tikzpicture}%
}
\hspace{0.15in}
\subfloat{
%
%
\begin{tikzpicture}[scale=0.78]

\begin{axis}[%
scale only axis,
xmode=log,
xmin=1000,
xmax=10000000,
xminorticks=true,
xlabel style={font=\Large\color{white!15!black}},
xlabel={$N_{V}$},
xticklabel style = {font=\large},
ymode=log,
ymin=0.01,
ymax=10,
yminorticks=true,
ylabel style={font=\Large\color{white!15!black}},
ylabel={Time (sec)},
yticklabel style = {font=\large},
axis background/.style={fill=white},
scaled ticks=true
]
\addplot [color=black, line width=1.0pt, mark=o, mark options={solid, black}, forget plot]
  table[row sep=crcr]{%
6852	0.0276900000002342\\
12168	0.035783000000265\\
27684	0.0613709999998946\\
108312	0.184366999999838\\
674664	0.813786000000164\\
1198728	1.38413400000013\\
2695248	3.2037590000001\\
};
\end{axis}

\end{tikzpicture}%
}
\caption{For the Neumann problem in \cref{ss:exintneu},
the left-hand plot
shows the relative $\ell_{\infty}$ error (solid line with solid markers) and
the relative $\ell_{2}$ error (dotted line with open markers) using our solver.
The dashed line is for reference, with slope corresponding to 
a convergence of order $O\!\left(\Delta x ^{3}\right)$.
The right-hand plot shows the total time for all steps (excluding mesh generation) for solving and evaluating the solution at all volumetric quadrature nodes.}
\label{fig:exintneu_conv_and_time}
\end{figure}
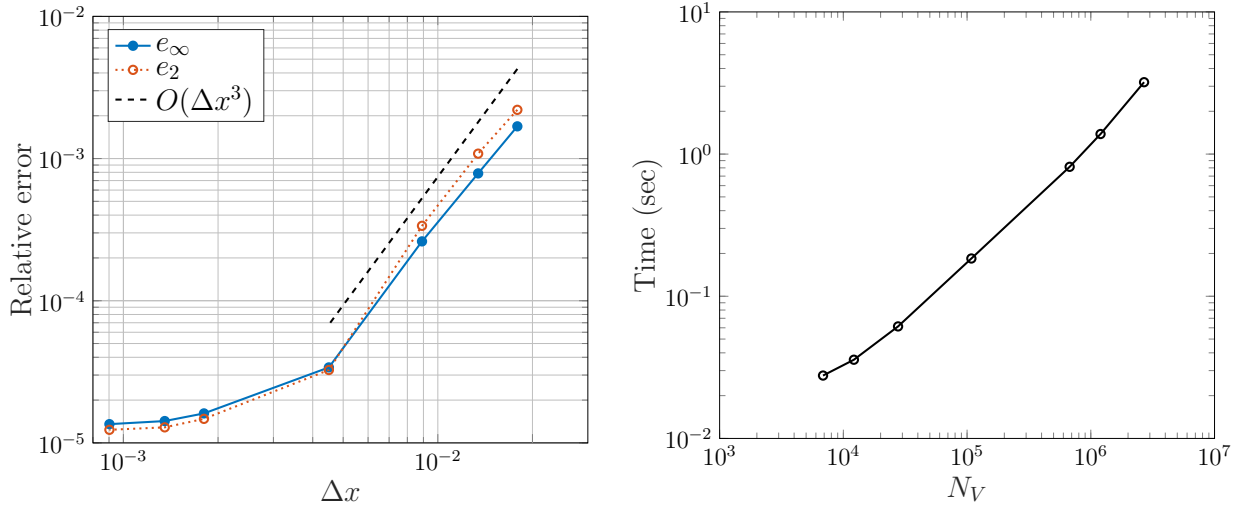
\subsection{An exterior Dirichlet problem}\label{ss:exextdir}
In our next example, we let
$\partial\Omega$ be defined by \cref{eq:bdryparam}, with nonzero coefficients $r=1$, $\alpha_{0}=1$, $\alpha_{5} = 3/10$, and $s=1$. (See \cref{fig:exintext_omega_and_sol}.)

In $\mathbb{R}^{2}\setminus \bar{\Omega}$, we let the source density be given by
\begin{equation}\label{eq:exextsourcedens}
f(\bm{x}) = 4\sum\limits_{i=1}^{3}e^{-\frac{(x_{1}-a_{i})^2 + (x_{2}-b_{i})^2}{\sigma_{i}}}\left(\frac{(x_{1}-a_{i})^2 + (x_{2}-b_{i})^2}{\sigma_{i}^{2}}-\frac{4}{\sigma_{i}}\right),
\end{equation}
with $a_{1}= 1/2$, $b_{1}=-1/2$, $\sigma_{1}=1/20$, $a_{2}= 1/2$, $b_{2}=1/2$, $\sigma_{2}=1/40$, $a_{3}= -3/10$, $b_{3}=0$, and $\sigma_{3}=1/15$. On the boundary 
we assume $u(\bm{x}) = g(\bm{x})$, corresponding to the exact solution
\begin{equation}\label{eq:exextsolana}
  u(\bm{x}) = \sum\limits_{i=1}^{3}e^{-\frac{(x_{1}-a_{i})^2 + (x_{2}-b_{i})^2}{\sigma_{i}}} + 10\log((x_{1}-a_{4})^2 + (x_{2}-b_{4})^2) \, ,
\end{equation}
with $a_{4}=3/10$, and $b_{4} = -0.5$.
Following the discussion in \cref{s:bie},
we represent the solution as
\[
    u(\bm{x}) = \mathcal{V}[f](\bm{x}) + \mathcal{D}_E[\mu](\bm{x})
+ 20 \log \| \bm{x} \|
\quad \text{for } \bm{x}\in \mathbb{R}^{2}\setminus \bar{\Omega}.
\]
The source term is exponentially decaying,
and it is easy to check that $|f(\bm{x})|< 10^{-16}$ for $\bm{x}$ outside the
domain $D = [-1.95,1.95]^{2}\,\setminus \bar{\Omega}$, 
shown in \cref{fig:exintext_omega_and_sol}.

\begin{figure}
\centering 
\subfloat{\includegraphics[scale=0.47]{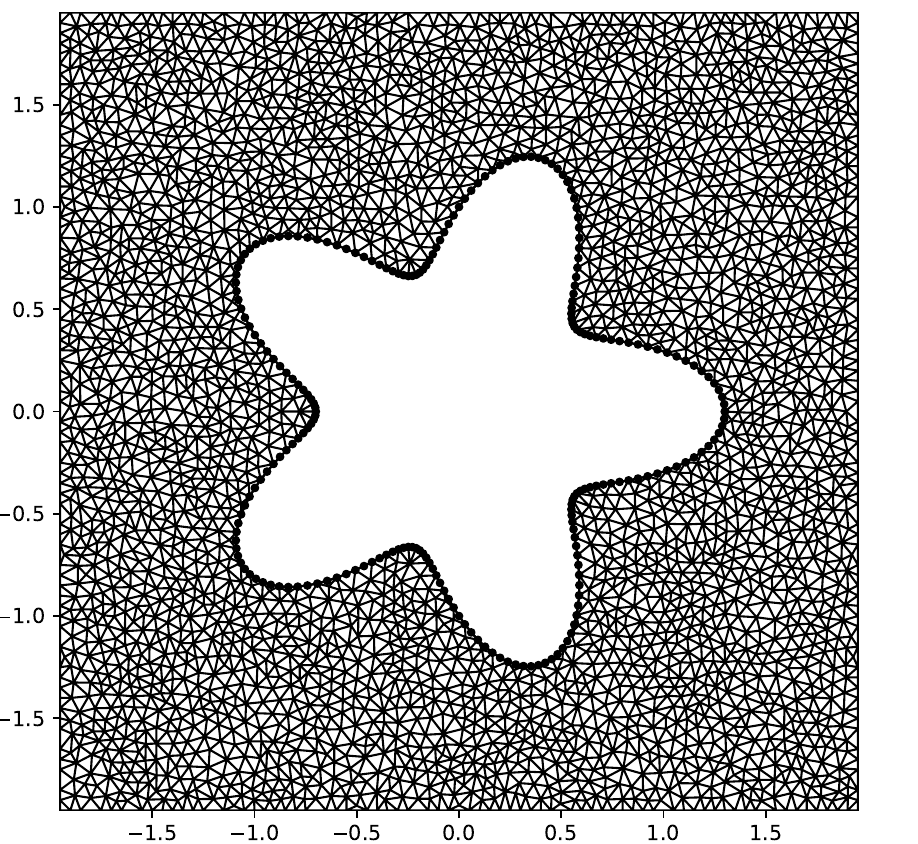}}
\hspace{0.15in}
\subfloat{\includegraphics[trim={4.5cm 8.8cm 5cm 8cm},clip,scale=0.65]{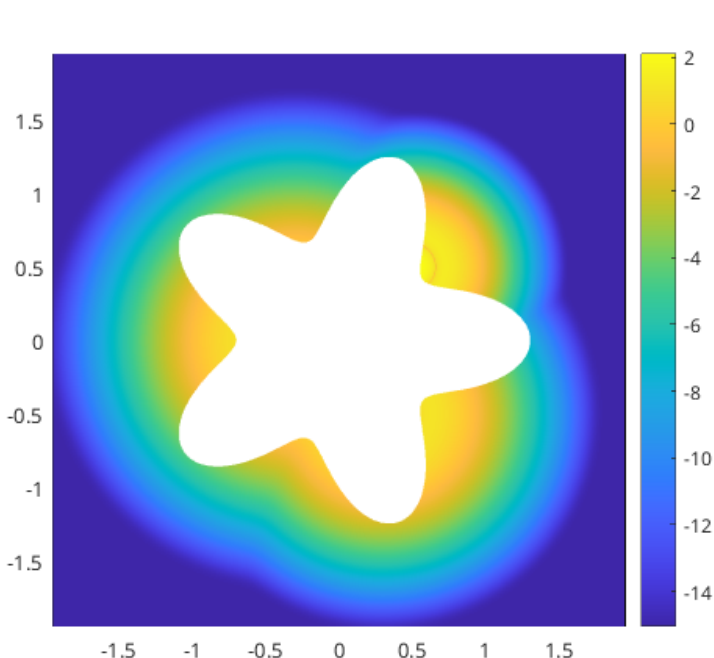}}
\caption{A volumetric mesh for the domain $D$ (left) in
\cref{ss:exextdir}, with
$36318$ elements and $h\approx 0.1$. There are $228$ nodes along the boundary 
$\partial\Omega$, marked with black dots. The right-hand figure shows 
the modulus of the source density \cref{eq:exextsourcedens} on a logartihmic
scale (base 10).}
\label{fig:exintext_omega_and_sol}
\end{figure}

We solve the exterior Dirichlet problem with
six different meshes, corresponding to $h=0.1$, $0.075$, $0.05$, $0.025$, $0.01$,
$0.0075$, using
three levels of dyadic refinement to improve the accuracy of the local
parts of the layer potentials, as described in \cref{s:dyadicref}.
Keeping terms up to $O(\delta)$ in the asymptotic expansions, we
expect third order accuracy in $\Delta x$. This is verified
in \cref{fig:exintext_conv_and_time}. The $l_{\infty}$ error saturates at 
about $3\times 10^{-6}$, which is more or less the requested tolerance.

Note from \cref{fig:exintext_conv_and_time} that the solver operates at about 
one million points per second per core for the four rightmost data points:
$9.6\times 10^{5}$ for $N_{V}=146136$, $1.1\times 10^{6}$ for $N_{V} = 583962$,  $1.1\times 10^{6}$ for $N_{V} = 3657276$, and $9.9\times 10^{5}$ for $N_{V} = 6491154$.

\begin{figure}
\centering 
\subfloat{
%
%
\definecolor{mycolor1}{rgb}{0.00000,0.44700,0.74100}%
\definecolor{mycolor2}{rgb}{0.85000,0.32500,0.09800}%
\begin{tikzpicture}[scale=0.78]

\begin{axis}[%
scale only axis,
xmode=log,
xmin=0.001,
xmax=0.03,
xminorticks=true,
xlabel style={font=\Large\color{white!15!black}},
xlabel={$\Delta x$},
xticklabel style = {font=\large},
extra x ticks={0.02},
extra x tick labels={},
ymode=log,
ymin=1e-08,
ymax=0.001,
yminorticks=true,
ylabel style={font=\Large\color{white!15!black}},
ylabel={Relative error},
yticklabel style = {font=\large},
axis background/.style={fill=white},
xmajorgrids,
xminorgrids,
ymajorgrids,
yminorgrids,
legend style={font=\Large,at={(0.03,0.97)}, anchor=north west, legend cell align=left, align=left, draw=white!15!black},
clip mode=individual,
]
\addplot [color=mycolor1, line width=1.0pt, mark=*, mark options={solid, mycolor1}, mark size=2.0pt]
  table[row sep=crcr]{%
1.812025088816563e-02     3.235942541312149e-04
     1.353105110205596e-02     1.266563760014896e-04\\
     9.033938778499824e-03     3.352874877532253e-05\\
     4.519301316530315e-03     4.641970466386256e-06\\
     1.805871986050080e-03     3.100657586298427e-06\\
     1.355517292747882e-03     3.041786621821009e-06\\
};
\addlegendentry{$e_{\infty}$}

\addplot [color=mycolor2, dotted, line width=1.0pt, mark=o, mark options={solid, mycolor2}, mark size=2.0pt]
  table[row sep=crcr]{%
     1.812025088816563e-02     3.027491775684975e-05\\
     1.353105110205596e-02     1.201849732857813e-05\\
     9.033938778499824e-03     3.404672873843637e-06\\
     4.519301316530315e-03     4.085512800255691e-07\\
     1.805871986050080e-03     7.928012219625174e-08\\
     1.355517292747882e-03     6.315117119996173e-08\\
};
\addlegendentry{$e_{2}$}

\addplot [color=black, dashed, line width=1.0pt]
  table[row sep=crcr]{%
     1.812025088816563e-02     4.488201610345445e-05\\
     1.353105110205596e-02     1.868849550275591e-05\\
     9.033938778499824e-03     5.561746202928457e-06\\
     4.519301316530315e-03     6.962955699579384e-07\\
};
\addlegendentry{$O(\Delta x ^{3})$}
\addplot [color=black, dashed, line width=1.0pt]
  table[row sep=crcr]{%
     1.812025088816563e-02     4.488201610345446e-04\\
     1.353105110205596e-02     1.868849550275592e-04\\
     9.033938778499824e-03     5.561746202928458e-05\\
     4.519301316530315e-03     6.962955699579385e-06\\
     };
\end{axis}

\end{tikzpicture}%
}
\hspace{0.15in}
\subfloat{
%
%
\begin{tikzpicture}[scale=0.78]

\begin{axis}[%
scale only axis,
xmode=log,
xmin=10000,
xmax=10000000,
xminorticks=true,
xlabel style={font=\Large\color{white!15!black}},
xlabel={$N_{V}$},
xticklabel style = {font=\large},
ymode=log,
ymin=0.01,
ymax=10,
yminorticks=true,
ylabel style={font=\Large\color{white!15!black}},
ylabel={Time (sec)},
yticklabel style = {font=\large},
axis background/.style={fill=white},
scaled ticks=true
]
\addplot [color=black, line width=1.0pt, mark=o, mark options={solid, black}, forget plot]
  table[row sep=crcr]{%
     3.631800000000000e+04     5.174600000009377e-02\\
     6.513600000000000e+04     1.471339999989141e-01\\
     1.461360000000000e+05     1.475669999999809e-01\\
     5.839620000000000e+05     5.378430000018852e-01\\
     3.657276000000000e+06     3.156111999999666e+00\\
     6.491154000000000e+06     6.467032999998082e+00\\
};
\end{axis}

\end{tikzpicture}%
}
\caption{For the exterior Dirichlet problem in \cref{ss:exextdir}, the
left-hand plot shows the relative $\ell_{\infty}$ error
(solid line with solid markers) and the relative $\ell_{2}$ error 
(dotted line with open markers). The dashed line is for reference, with slope
corresponding to a convergence rate of 
order $O\!\left(\Delta x ^{3}\right)$. The right-hand plot shows the total time for all steps (excluding mesh generation) for solving and evaluating the solution at all volumetric quadrature nodes.}
\label{fig:exintext_conv_and_time}
\end{figure}
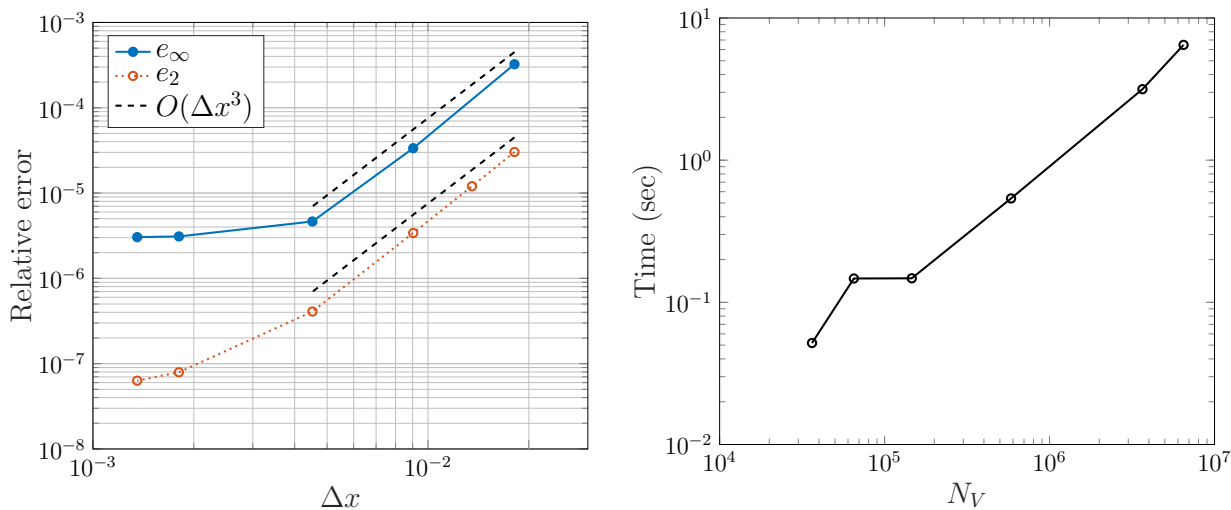

\subsection{Robustness With Respect to Mesh Quality}\label{ss:exrobust}
A major advantage of potential theory in solving the Poisson equation
is that the volume integral $\calV[f]$ expresses a particular solution 
by quadrature. This leads to both robust performance and straightforward error 
analysis.  To see this, 
let $\tilde{f}(\bm{x})$ be a piecewise polynomial approximation of 
${f}(\bm{x})$ and let
${\tilde Q}[\tilde{f}](\bm{x})$ denote the quadrature approximation
of $\mathcal{V}[\tilde{f}](\bm{x})$. Then,
as noted in \cite{langston2011}, 
the total error can be written as
\begin{equation}
\begin{aligned}
e(\bm{x}) &= | \mathcal{V}[f](\bm{x}) - {\tilde Q}[\tilde{f}](\bm{x}) | \\
         & \leq
\left| \int\limits_{\Omega} G(\bm{x}-\bm{x'}) 
(f(\bm{x'}) - \tilde{f}(\bm{x'}))\,\mathrm{d}\bm{x'} \right| +
\left| \int\limits_{\Omega} G(\bm{x}-\bm{x'}) 
\tilde{f}(\bm{x})\,\mathrm{d}\bm{x'} - {\tilde Q}[\tilde{f}](\bm{x}) \right| \\
&= O(\|f(\bm{x})-\tilde{f}(\bm{x}) \|_\infty) + O(\epsilon) \, .
\end{aligned}
\label{volint_error}
\end{equation}
The first term in this estimate is the error in the approximation of the source distribution,
and follows immediately from the fact that $\mathcal{V}$ is a bounded operator.
The second term represents the quadrature error in evaluating the volume integral
for the {\em approximate} right-hand side $\tilde{f}(\bm{x})$.
PDE-based methods require much more complicated estimates that depend on the 
smoothness of the (unknown) solution itself.

As a demonstration, consider the interior Dirichlet problem in the domain
$\Omega$ from \cref{ss:exextdir}. The source density is given by 
\cref{eq:exintneusourcedens}, and the Dirichlet boundary data by 
\cref{eq:neusolana}. 
We solve the problem for two sequences of mesh refinement: 
the first consists of seven meshes, as described in the beginning of 
\cref{s:numericalresults}. 
In the second case, we repeat the refinement process but, 
for every triangle, we pick one of the three sides, say $S$ at random, 
excluding the curved boundary segments, and pick a point $P$ at random along 
the line containing $S$. 
If $P$ ends up outside the triangle, we move it to the closest 
endpoint of $S$. We then split the triangle along the line drawn from 
the random point to the corner opposite $S$. 
This results in a nonconforming mesh, with elements of arbitrarily poor 
aspect ratio, of which about $4\%$ are actually degenerate 
(see \cref{fig:exrobustmesh} 
for an example corresponding to $h=0.1$). 
We use the same relation between $\delta$ and the average mesh spacing
$\Delta x^{2}$ as in our other examples.
\begin{figure}
\centering
\begin{tikzpicture}[spy using outlines={rectangle,red,magnification=4,size=3.5cm, connect spies}]
\node {\pgfimage[interpolate,height=7cm]{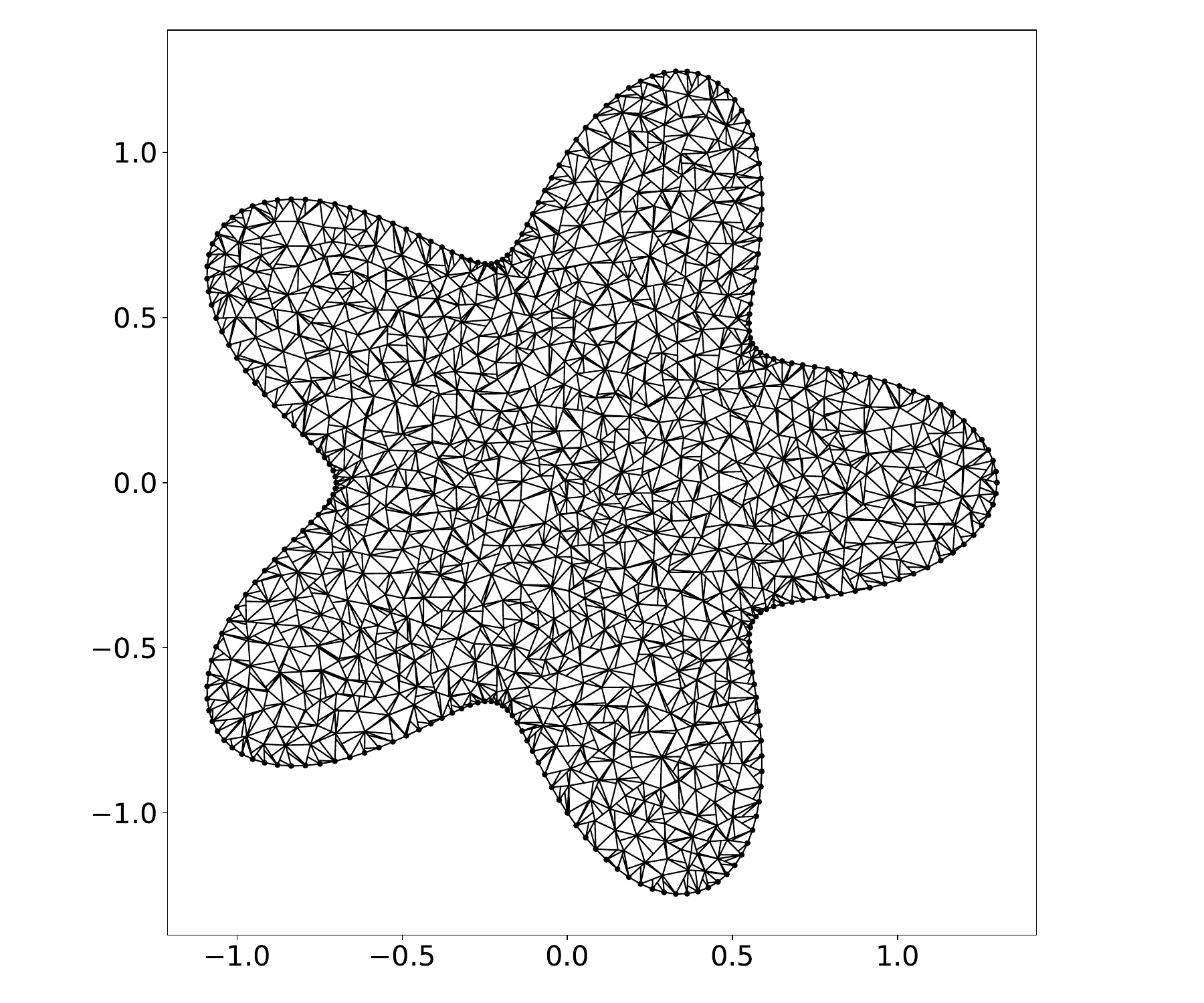}};
\spy on (1,0.3) in node [left] at (7,1.25);
\end{tikzpicture}
\caption{Left: A degenerate mesh (left) obtained from the conforming triangulation
created by {\tt Triangle} \cite{trianglepaper,trianglelib} with $h=0.1$. 
The parametrization of the boundary is given in \cref{ss:exextdir}. 
Of the $3624$ new triangles, $127$ are degenerate. A blow-up of a small region 
is shown on the right.}
\label{fig:exrobustmesh}
\end{figure}

We observe in \cref{fig:exrobust_conv} that mesh distortion had a negligible effect on the error in the solution.
Moreover, removing triangles (creating gaps in the discretization)
induces a small error as well $-$ proportional to the area fraction of
triangles removed.

\begin{figure}[h]
\centering 
%
%
\definecolor{mycolor1}{rgb}{0.00000,0.44700,0.74100}%
\definecolor{mycolor2}{rgb}{0.85000,0.32500,0.09800}%
\begin{tikzpicture}[scale=0.78]

\begin{axis}[%
scale only axis,
xmode=log,
xmin=1e-6,
xmax=0.003,
xminorticks=true,
xlabel style={font=\Large\color{white!15!black}},
xlabel={$\delta$},
extra x ticks={0.002},
extra x tick labels={},
xticklabel style = {font=\large},
ymode=log,
ymin=1e-08,
ymax=0.003,
yminorticks=true,
ylabel style={font=\Large\color{white!15!black}},
ylabel={Relative error},
yticklabel style = {font=\large},
axis background/.style={fill=white},
xmajorgrids,
xminorgrids,
ymajorgrids,
yminorgrids,
legend style={font=\Large,at={(0.03,0.97)}, anchor=north west, legend cell align=left, align=left, draw=white!15!black},
clip mode=individual,
]
\addplot [color=mycolor1, line width=1.0pt, mark=*, mark options={solid, mycolor1}, mark size=2.0pt]
  table[row sep=crcr]{%
     9.060379704474429e-04     5.751918073542443e-05\\
     5.228085044004151e-04     2.837455721621307e-05\\
     2.347745798855494e-04     9.370096943038349e-06\\
     6.011487245790398e-05     1.583517743727053e-06\\
     9.710857461079579e-06     8.639670272700639e-07\\
     5.476650070427052e-06     8.177329070598145e-07\\
     2.439060356200125e-06     1.074911344509009e-06\\
};

\addplot [color=mycolor2, dotted, line width=1.0pt, mark=o, mark options={solid, mycolor2}, mark size=2.0pt]
  table[row sep=crcr]{%
     9.060379704474429e-04     1.056229717482648e-05\\
     5.228085044004151e-04     4.308129059030584e-06\\
     2.347745798855494e-04     1.289057084765961e-06\\
     6.011487245790398e-05     4.727182728941369e-07\\
     9.710857461079579e-06     9.871963679464511e-08\\
     5.476650070427052e-06     9.626944932492544e-08\\
     2.439060356200125e-06     9.465263375682426e-08\\
};
\addplot [color=mycolor1, line width=1.0pt, mark=square*, mark options={solid, mycolor1}, mark size=2.0pt]
  table[row sep=crcr]{%
     9.060379704474449e-04     7.973889932925468e-05\\
     5.228085044004164e-04     3.634742411910547e-05\\
     2.347745798855498e-04     1.150546410135194e-05\\
     6.011487245790397e-05     1.571959278818947e-06\\
     9.710857461079581e-06     1.081793863973289e-06\\
     5.476650070427052e-06     1.079304478629404e-06\\
     2.439060356200125e-06     1.070292899523101e-06\\
};

\addplot [color=mycolor2, dotted, line width=1.0pt, mark=square, mark options={solid, mycolor2}, mark size=2.0pt]
  table[row sep=crcr]{%
     9.060379704474449e-04     1.066619380494925e-05\\
     5.228085044004164e-04     4.307149939723041e-06\\
     2.347745798855498e-04     1.287671339463834e-06\\
     6.011487245790397e-05     4.771734810266285e-07\\
     9.710857461079581e-06     9.840691074091229e-08\\
     5.476650070427052e-06     9.593914360271627e-08\\
     2.439060356200125e-06     9.456357000868728e-08\\
};

\addplot [color=black, dashed, line width=1.0pt]
  table[row sep=crcr]{%
     9.060379704474450e-04     1.720452232222099e-04\\
     5.228085044004167e-04     7.541139705766146e-05\\
     2.347745798855497e-04     2.269340319540672e-05\\
     6.011487245790398e-05     2.940328925231237e-06\\
   };

\addplot [color=black, dashed, line width=1.0pt]
  table[row sep=crcr]{%
     9.060379704474450e-04     1.720452232222099e-05\\
     5.228085044004167e-04     7.541139705766145e-06\\
     2.347745798855497e-04     2.269340319540672e-06\\
};
\end{axis}

\end{tikzpicture}%
\caption{Relative errors, as functions of $\delta$ for numerical experiments in \cref{ss:exrobust}: solid lines with solid markers are the relative $\ell_{\infty}$ errors, dotted lines with open markers are the relative $\ell_{2}$ errors. Round markers ($\bullet$, $\circ$) correspond to errors for the nondegenerate meshes. Square markers ($\mathsmaller{\blacksquare}$, $\mathsmaller{\square}$) correspond to errors for the degenerate meshes. The dashed line without markers is a reference for the expected order of convergence $O\!\left(\delta^{3/2}\right)$.}
\label{fig:exrobust_conv}
\end{figure}
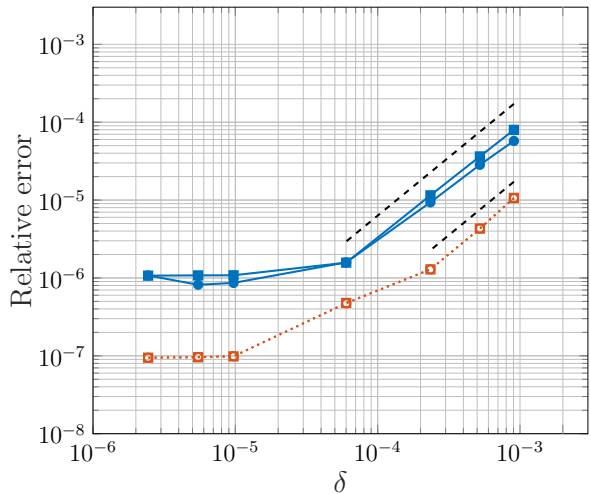
\section{Discussion}\label{s:discussion}
We have presented a new class of fast algorithms for harmonic volume and layer
potentials that is based on evaluating heat potentials to a steady state limit,
separating the calculation into a `` far history" part, a ``near history"
part and a ``local" part.
The first two are evaluated using the NUFFT and the latter is 
computed using asymptotics (or more precisely, a combination of asymptotics and 
a telescoping localized quadrature rule). For layer potentials,
it is closely related to the method of \cite{gs:fhp,strainfpt} and to 
Ewald-like methods, which split the harmonic Green's function into a regularized
far-field kernel with a near field correction 
\cite{klinteberg2017,bealetlupova2024,bealetlupova2024b,beale2016,cortez2001,cortez2005,gs:fhp,dmk,li2009sisc,palssontornberg,shamshirgar2021jcp,strainfpt}. 

Three novel features have been introduced here: 
\begin{enumerate}
\item the construction of accurate asymptotic
expansions for volume potentials when the volumetric source distribution is
smooth within a domain $\Omega$, but discontinuous as a function in the 
ambient space (here $\mathbb{R}^2$),
\item the introduction of several steps of dyadic refinement 
to obtain arbitrary precision in evaluating the local part,
\item the separation of the history part into the ``near" and ``far" components
which provides approximately a three fold speedup in the NUFFT computations
(used also in \cite{dmk}).
\end{enumerate}

The second feature is of significant practical importance, since 
high order asymptotic formulas become more and more ill-conditioned, as they 
involve higher derivatives of the density or the curve itself
(with catastrophic cancellation reducing their effectiveness). 
In practice, we recommend
using the formulas of order no greater than $O(\delta^{3/2})$.
Thus, when higher precision is requested, it is critical to be able to 
carry out $j$ refinement steps to make the value 
$\delta^\ast = \delta/(4^j)$ sufficiently small before invoking the asymptotic 
approximation. This is described in \cref{s:dyadicref} and
related to the hybrid asymptotic/numerical quadrature method of 
\cite{wang2019acom}.

From a user's perspective, perhaps the 
most important feature of the method is the extent to which it is agnostic
to the details of the underlying discretization. All that we require
is the set of weights and nodes for a smooth quadrature rule in the domain and
the set of weights and nodes for a smooth quadrature rule on the boundary.
This is compatible with unstructured triangulations, adaptive 
quad and octtrees, spectral elements,
and various kinds of composite grids
\cite{shravan_hai_volint,composite2012,berger2017,overture,functionintension,saye2015}. 
As demonstrated in our numerical examples, the integral formulation helps
make it insensitive to flaws in the triangulation, including outright degeneracies
and gaps, while maintaining extremely high throughput.

Of potential mathematical interest is that,
by taking the heat potential viewpoint, 
our asymptotic formulas can be used to derive 
the jump properties of the corresponding elliptic layer
potentials across the boundary quite naturally (as seen in 
\eqref{eq:dlplocO5} and \eqref{eq:dlplocO5onsurf}). This provides an
alternative derivation to the ones typically found in the literature
\cite{guenther1988partial,kelloggpot,KressRainer2014LIE,mikhlin,stakgold}.

The algorithm presented above is suitable only for quasi-uniform 
triangulations, as it relies on the NUFFT for speed, and the balance 
between the history and local parts would be broken by multiscale
triangulation.
Fortunately, the tools developed here, including asymptotics 
and local dyadic refinement, can be integrated naturally 
into the DMK framework of \cite{dmk} to develop a fully adaptive 
algorithm. We will report on that
implementation, the extension to other
Green's functions, as well as three-dimensional experiments, at a later date.

\section{Acknowledgements}
The authors would like to thank Fabrizio Falasca at the Courant Institute of Mathematical Sciences, New York University, and Manas Rachh and Dan Fortunato 
at the Center for Computational Mathematics, Flatiron Institute, for helpful discussions. The first author gratefully
acknowledges support from the Knut and Alice Wallenberg Foundation under
grant 2020.0258. 
The work of the second and fourth authors was partially supported by the Office of Naval Research under award N00014-18-1-2307.
\bibliographystyle{siam}
\bibliography{journalnames,references}
\end{document}